\documentclass[a4paper,11pt]{article}
\usepackage{amsmath, amsfonts, amscd, amssymb, amsthm, enumerate, xypic}

\def\bfB{\mathbf{B}}

\newcommand{\Mat}{\operatorname{M}}

\newcommand{\id}{\operatorname{id}}

\newcommand{\Ker}{\operatorname{Ker}}

\newcommand{\Vect}{\operatorname{span}}
\newcommand{\im}{\operatorname{Im}}

\newcommand{\rk}{\operatorname{rank}}

\renewcommand{\setminus}{\smallsetminus}

\def\K{\mathbb{K}}

\def\Z{\mathbb{Z}}

\renewcommand{\L}{\mathbb{L}}

\def\calH{\mathcal{H}}

\def\calL{\mathcal{L}}

\def\calQ{\mathcal{Q}}

\def\calS{\mathcal{S}}
\def\calT{\mathcal{T}}

\def\calZ{\mathcal{Z}}

\def\lcro{\mathopen{[\![}}
\def\rcro{\mathclose{]\!]}}

\theoremstyle{definition}
\newtheorem{Def}{Definition}
\newtheorem{Not}[Def]{Notation}

\theoremstyle{plain}
\newtheorem{theo}{Theorem}[section]
\newtheorem{prop}[theo]{Proposition}
\newtheorem{cor}[theo]{Corollary}
\newtheorem{lemma}[theo]{Lemma}
\newtheorem{claim}{Claim}

\theoremstyle{plain}

\theoremstyle{remark}
\newtheorem{Rems}{Remarks}
\newtheorem{Rem}[Rems]{Remark}

\title{LDB division algebras}
\author{Cl\'ement de Seguins Pazzis\footnote{Universit\'e de Versailles Saint-Quentin-en-Yvelines, Laboratoire de Math\'ematiques
de Versailles, 45 avenue des Etats-Unis, 78035 Versailles cedex, France}
\footnote{e-mail address: dsp.prof@gmail.com}}

\begin{document}

\thispagestyle{plain}

\maketitle

\begin{abstract}

An LDB division algebra is a triple $(A,\star,\bullet)$ in which
$\star$ and $\bullet$ are regular bilinear laws on the finite-dimensional non-zero vector space $A$
such that $x \star (x \bullet y)$ is a scalar multiple of $y$ for all vectors $x$ and $y$ of $A$.
This algebraic structure has been recently discovered in the study of the critical case in
Meshulam and \v Semrl's estimate of the minimal rank in non-reflexive operator spaces.

In this article, we obtain a constructive description of all LDB division algebras over an arbitrary field
together with a reduction of the isotopy problem to the similarity problem for specific types of quadratic forms
over the given field. In particular, it is shown that the dimension of an LDB division algebra is always a power of $2$, and that
it belongs to $\{1,2,4,8\}$ if the characteristic of the underlying field is not $2$.
\end{abstract}



\section{Introduction}

Throughout the article, $\K$ denotes an arbitrary field.

\subsection{The main concept}

A division algebra (over $\K$) is a pair $(A,\star)$ in which $A$ is a non-zero finite-dimensional vector space (over $\K$)
and $\star : A \times A \rightarrow A$ is a bilinear map that is \emph{regular} in the sense that
$x \star y \neq 0$ for all non-zero vectors $x$ and $y$ of $A$.
This is the weakest possible definition of a finite-dimensional division algebra over a field
(we require neither associativity nor the existence of a one-sided
unity). Following the terminology introduced by Albert \cite{Albert}, we say that
two division algebras $(A,\star)$ and $(B,\bullet)$ are \textbf{isotopic}
whenever there are vector space isomorphisms $f$, $g$ and $h$ from $A$ to $B$ such that
$$\forall (x,y)\in A^2, \; x \star y=h^{-1}(f(x)\bullet g(y)).$$

If the field $\K$ is algebraically closed, then a division algebra over $\K$ must be of dimension $1$. Over the reals, a division algebra
must have dimension $1$, $2$, $4$ or $8$
(see \cite{BottMilnor,Kervaire}). This result extends to real closed fields \cite{DDH,Dieterich}. Over fields that are
neither algebraically closed nor real closed, division algebras exist in arbitrarily large dimensions.
Apart from that, little is known in general over arbitrary fields without additional assumptions on
the multiplicative law.

In this article, we shall focus on a special type of division algebra that has been discovered very recently
\cite{dSPLLD1}. Given a division algebra $(A,\star)$,
we define a \textbf{quasi-left-inversion} for $\star$ as a map $\bullet  : A \times A \rightarrow A$ such that, for all $(x,y)\in (A \setminus \{0\})^2$,
 the vector $x \bullet y$ is non-zero and $x \star (x \bullet y)$ is a scalar multiple of $y$
 (in other words, up to a scalar, $x \bullet y$ is the solution of the equation $x \star z=y$ (with unknown $z$), that is ``$y$ left-divided by $x$").
If we have a bilinear quasi-left-inversion $\bullet$ for $\star$, then one easily obtains (see Proposition 5.1 of \cite{dSPLLD1})
a uniquely defined quadratic form $q$ on $A$ such that
$$\forall (x,y)\in A^2, \quad x \star (x \bullet y)=q(x)\,y.$$
The quadratic form $q$ is anisotropic (that is $q(x) \neq 0$ for all non-zero $x \in A$),
but its polar form $b_q : (x,y) \mapsto q(x+y)-q(x)-q(y)$ might be degenerate if the field $\K$ has characteristic $2$
(a mundane example is given by the standard multiplication on $A=\K$).

In the above situation, one proves that $\bullet$ is, up to multiplication by a non-zero scalar,
the unique bilinear quasi-left-inversion for $\star$ (see Proposition 5.2 of \cite{dSPLLD1}).
The triple $(A,\star,\bullet)$ is then called a \textbf{left-division-bilinearizable division algebra}
(in abbreviated form: LDB division algebra) and $q$ is called the quadratic form attached to it.
Note that, for all $\lambda \in \K^*$, the triple $(A,\star,\lambda\bullet)$ is
another LDB division algebra with $\lambda q$ as its attached quadratic form.
Whenever possible, we shall understate the two laws $\star$ and $\bullet$ and simply say that $A$ is an LDB division algebra.

When one is confronted with quadratic forms in the context of division algebras the comparison with
composition algebras is unavoidable. Let us recall that a (finite-dimensional) \textbf{composition algebra} is a triple $(A,\star,N)$
in which $A$ is a finite-dimensional non-zero vector space over $\K$, $\star : A^2 \rightarrow A$ is a bilinear map, and
$N$ is a non-degenerate quadratic form on $A$ such that
$$\forall (x,y) \in A^2, \; N(x \star y)=N(x)N(y).$$
A composition algebra is called a \textbf{Hurwitz algebra} when it has a (two-sided) unity.

Two division algebras $(A,\star)$ and $(B,\star')$ are called isomorphic
when there is a vector space isomorphism $f : A \overset{\simeq}{\rightarrow} B$ such that $\forall (x,y)\in A^2, \; f(x \star y)=f(x) \star' f(y)$.
A famous result \cite{Hurwitz,Involutions}\footnote{A good account of the history of Hurwitz's result is given in
chapter 10 of \cite{Zahlen}.}
states that every Hurwitz algebra has dimension $1$, $2$, $4$ or $8$
and is isomorphic to one of the following canonical Hurwitz algebras:
\begin{itemize}
\item the one-dimensional Hurwitz algebra $(\K,\cdot ,x \mapsto x^2)$, if $\K$ does not have characteristic $2$;
\item the two-dimensional Hurwitz algebra $(\K \times \K, \cdot, (x,y) \mapsto xy)$;
\item the two-dimensional Hurwitz algebra $(\L,\cdot,N_{\L/\K})$, where $\L$ is a separable quadratic extension of $\K$,
and $N_{\L/\K}$ is the norm of $\L$ over $\K$;
\item the four-dimensional Hurwitz algebra $(C(q),\cdot,N_{C(q)})$, where $q$ is a regular $2$-dimensional quadratic form over $\K$, and
$C(q)$ denotes its Clifford algebra with norm denoted by $N_{C(q)}$;
\item the eight-dimensional Hurwitz algebra $(C(q)^2,\times_\varepsilon,N_\varepsilon)$, where $q$ is a regular $2$-dimensional quadratic form over $\K$,
$\varepsilon$ is a non-zero scalar, and $\times_\varepsilon$ and $N_\varepsilon$ are defined, respectively, by
$$(a,b) \times_\varepsilon (c,d):=(ac-d\overline{b}\, ,\, \overline{a} d-\varepsilon cb)$$
and
$$N_\varepsilon(a,b):=a\overline{a}-\varepsilon b \overline{b},$$
where $x \mapsto \overline{x}$ denotes the conjugation in the quaternion algebra $C(q)$.
\end{itemize}
Let $(A,\star,N)$ be a Hurwitz algebra. One sees that $(A,\star)$ is a division algebra
if and only if $N$ is anisotropic. Moreover, if $N$ is anisotropic then we can find a bilinear quasi-left-inversion for $\star$.
Indeed, denoting by $x \mapsto \overline{x}$
the opposite of the reflection of the quadratic space $(A,N)$ along the unity of $(A,\star)$,
one can check - e.g., by referring to the above canonical situations - that the following identity holds:
$$\forall (x,y)\in A^2, \; x \star (\overline{x} \star y)=N(x)\,y.$$
Thus, provided that $N$ is anisotropic, the law $\bullet : (x,y)\mapsto \overline{x} \star y$ is a bilinear quasi-left-inversion for $\star$
and the triple $(A,\star,\bullet)$ is an LDB division algebra with attached quadratic form $N$.
In that situation, we shall say that it is an LDB division algebra of \textbf{Hurwitz type}.
More precisely, we shall say that it is of \textbf{separable quadratic type}, \textbf{quaternionic type}, or \textbf{octonionic type}, depending on whether
it has dimension $2$, $4$ or $8$.

Assuming now that $\K$ has characteristic $2$, we can give an additional kind of example.
A finite-dimensional field extension $\L$ of $\K$
is called \textbf{hyper-radicial} if $\forall x \in \L, \; x^2 \in \K$.
Given such an extension, $q : x \mapsto x^2$ is an anisotropic quadratic form on $\L$ seen as a vector space over $\K$. Thus, with
$\star$ as the multiplication of the field $\L$, the triple $(\L,\star,\star)$ is an LDB division algebra with $q$
as its attached quadratic form: we call it the LDB division algebra associated with the hyper-radicial extension $\L$,
and we say that it is an \textbf{LDB division algebra of hyper-radicial type.}
Note that $q$ is totally degenerate, i.e.\ its polar form is zero.
Moreover, in this situation we see that the degree of $\L$ over $\K$ is a power of $2$, and if
$\dim_\K \L=2^n$ and $\L=\K[a_1,\dots,a_n]$, then $q \simeq \langle 1,a_1^2\rangle \otimes \cdots \otimes \langle 1,a_n^2\rangle$.

\vskip 3mm
Here is a unification of some of the above examples. First of all $(\K,\cdot,\cdot)$ is always an LDB division algebra
with attached quadratic form $x \mapsto x^2$: it is of Hurwitz type if $\K$ has characteristic not $2$,
otherwise it is of hyper-radicial type (take $\L=\K$).
Next, let $\L$ be a quadratic field extension of $\K$, and denote by $x \mapsto \overline{x}$
the non-identity automorphism of $\L$ over $\K$ if $\L$ is a separable extension of $\K$, and the identity of $\L$ otherwise.
Denoting by $\star$ the multiplication on $\L$ and defining $\bullet$ by $x \bullet y:= \overline{x} \star y$,
one sees that $(\L,\star,\bullet)$ is an LDB division algebra whose attached quadratic form is the norm of $\L$ over $\K$;
it is of Hurwitz type if $\L$ is a separable extension of $\K$, otherwise $\K$ has characteristic $2$ and
$(\L,\star,\bullet)$ is of hyper-radicial type; in any case we shall say that $(\L,\star,\bullet)$ is an LDB division algebra
of \textbf{quadratic type.}

\vskip 3mm
LDB division algebras were recently discovered in the study of non-reflexive spaces of linear operators.
Recall that, given vector spaces $U$ and $V$, a linear subspace $\calS$ of the space $\calL(U,V)$ of all linear maps from $U$ to $V$
is called \textbf{(algebraically) reflexive} when every $f \in \calL(U,V)$ that satisfies $\forall x \in U, \; f(x) \in \calS x$ belongs to $\calS$.
A result of Meshulam and \v Semrl \cite{MeshulamSemrlLAA} states that, provided that the underlying field has more than $n$ elements,
a non-reflexive $n$-dimensional operator space must contain a non-zero operator $f$ with $\rk(f) \leq 2n-2$
(it was recently shown that the provision on the cardinality of the underlying field is unnecessary \cite{dSPminrank}).
In \cite{dSPLLD1}, investigating the optimality of this result has led to the discovery of LDB division algebras
and their connection to examples in which the upper-bound $2n-2$ is reached:
take an $n$-dimensional LDB division algebra $(A,\star,\bullet)$ with attached quadratic form $q$, and consider the bilinear map
$$\Gamma_A : \begin{cases}
(A \oplus \K^2) \times A^2 & \longrightarrow A^2 \\
\bigl(x+(\lambda,\mu),(y,z)\bigr) & \longmapsto (x \star z+\lambda y\, ,\, x \bullet y+\mu z).
\end{cases}$$
The set $\calT_A$ consisting of all the endomorphisms $\Gamma_A(x+(\lambda,\mu),-)$ of $A^2$, with $(x,\lambda,\mu)\in A \times \K^2$,
is an $(n+2)$-dimensional linear subspace of $\calL(A^2)$ called the \textbf{twisted operator space}
attached to $(A,\star,\bullet)$. This operator space has very interesting properties:
given $(x,\lambda,\mu)\in A \times \K^2$, the endomorphism $\Gamma_A(x+(\lambda,\mu),-)$ is non-singular if and only if
$q(x)-\lambda \mu \neq 0$ (see \cite[Proposition 5.5]{dSPLLD1}). In other words, by identifying $\calT_A$ with $A \oplus \K^2$ through the isomorphism
$X \mapsto \Gamma_A(X,-)$, the set of all singular operators in $\calT_A$
is seen to correspond to the isotropy cone of the quadratic form $\widetilde{q} : x+(\lambda,\mu) \mapsto q(x)-\lambda \mu$.
Moreover, $\calT_A$ is \textbf{locally linearly dependent}, that is, for all $(y,z) \in A^2$,
there is a non-zero operator $f \in \calT_A$ such that $f(y,z)=0$ (see \cite[Proposition 5.4]{dSPLLD1}). Using this, one proves that
every anisotropic hyperplane of $\calT_A$ is non-reflexive \cite[Proposition 5.6]{dSPLLD1}.
Thus, if we have an anisotropic hyperplane of $(A\oplus \K^2, \widetilde{q})$,
then we obtain an $(n+1)$-dimensional non-reflexive operator space in which every non-zero operator has rank $2n=2(n+1)-2$,
thus yielding an example which demonstrates that the upper-bound of Meshulam and \v Semrl is optimal. Theorem
6.1 of \cite{dSPLLD1} shows that all the non-reflexive operator spaces for which Meshulam and \v Semrl's upper bound
is reached essentially arise from this construction, provided that the underlying field $\K$ be of large cardinality.

\subsection{The main results}

The purpose of this article is to provide a constructive description of all LDB division algebras. First, we need to define relevant notions
of isomorphisms for these structures. Let $(A,\star,\bullet)$ be an LDB division algebra, together with
three isomorphisms $f : B \overset{\simeq}{\rightarrow} A$, $g : B \overset{\simeq}{\rightarrow} A$ and $h : A \overset{\simeq}{\rightarrow} B$.
Then, the composition laws $\star'$ and $\bullet'$ on $B$ defined by
$$x \star' y=h\bigl(f(x)\star g(y)\bigr) \quad \text{and} \quad x \bullet' y=g^{-1}\bigl(f(x)\bullet h^{-1}(y)\bigr)$$
yield an LDB division algebra $(B,\star',\bullet')$ with attached quadratic form $x \mapsto q(f(x))$.
This motivates the following definition:

\begin{Def}
Let $(A,\star,\bullet)$ and $(B,\star',\bullet')$ be LDB division algebras. \\
We say that $(A,\star,\bullet)$ and $(B,\star',\bullet')$ are \textbf{weakly equivalent} when the
division algebras $(A,\star)$ and $(B,\star')$ are isotopic. \\
We say that $(A,\star,\bullet)$ and $(B,\star',\bullet')$ are \textbf{equivalent} when there are isomorphisms
$f : B \overset{\simeq}{\rightarrow} A$, $g : B \overset{\simeq}{\rightarrow} A$ and $h : A \overset{\simeq}{\rightarrow} B$
such that, for all $(x,y)\in B^2$,
$$x \star' y=h\bigl(f(x)\star g(y)\bigr) \quad \text{and} \quad x \bullet' y=g^{-1}\bigl(f(x)\bullet h^{-1}(y)\bigr).$$
\end{Def}

Note that the LDB division algebras $(A,\star,\bullet)$ and $(B,\star',\bullet')$ are weakly equivalent if and only if
there exists a non-zero scalar $\lambda \in \K^*$ for which $(A,\star,\bullet)$ and $(B,\star',\lambda\bullet')$ are equivalent.
Remember that two quadratic forms $q$ and $q'$ on respective vector spaces $V$ and $V'$ are \textbf{equivalent} (in which case we write $q \simeq q'$)
when there exists a vector space isomorphism $u : V \overset{\simeq}{\rightarrow} V'$ such that $q'(u(x))=q(x)$ for all $x \in V$;
they are called \textbf{similar} when there is a non-zero scalar $\lambda$ such that $q \simeq \lambda q'$.
It is then easily seen that equivalent (respectively, weakly equivalent)
LDB division algebras have equivalent (respectively, similar) attached quadratic forms.

From there, our aim is to relate LDB division algebras to known structures,
both for the relation of weak equivalence and for the one of equivalence.

We split our results into three theorems. We shall say that an LDB division algebra is \textbf{non-degenerate} (respectively, \textbf{degenerate})
when the attached quadratic form is non-degenerate (respectively, degenerate).
Note that, over a field of characteristic not $2$, every LDB division algebra is non-degenerate since the attached quadratic form is anisotropic.

\begin{theo}[Structure theorem for non-degenerate LDB division algebras]\label{nondegeneratetheo}
Every non-degenerate LDB division algebra has dimension $1$, $2$, $4$ or $8$. \\
Every non-degenerate LDB division algebra is weakly equivalent to an LDB division algebra of Hurwitz type. \\
Every non-degenerate LDB division algebra whose attached quadratic form represents $1$ is
equivalent to an LDB division algebra of Hurwitz type.
\end{theo}

Note in particular that, for fields of characteristic not $2$, the quadratic form attached to an LDB division algebra
is always similar to a \textbf{Pfister form}, that is a form of type $\langle 1,a_1\rangle \otimes \cdots \otimes \langle 1,a_n\rangle$
for some $(a_1,\dots,a_n)\in (\K^*)^n$.

\begin{Rem}
Theorem \ref{nondegeneratetheo} essentially states that the non-degenerate LDB division algebras are the isotopes
of the Hurwitz algebras that are division algebras.
This compares interestingly with the relationship between composition algebras and Hurwitz algebras.
Indeed, it is a rather elementary observation that the composition algebras are the \emph{orthogonal} isotopes of Hurwitz algebras in the following sense: two triples
$(A,\star,q)$ and $(B,\star',q')$ - in which $(A,\star)$ and $(B,\star')$ are non-associative algebras
and $q$ and $q'$ are quadratic forms, respectively, on $A$ and $B$ - are called orthogonally isotopic when there
are isometries $f$, $g$ and $h$ from $(A,q)$ to $(B,q')$ such that
$$\forall (x,y)\in A^2, \; x \star y=h^{-1}\bigl(f(x) \star' g(y)\bigr).$$
\end{Rem}

\vskip 3mm
Here is our result on degenerate LDB division algebras over fields of characteristic $2$:

\begin{theo}[Structure theorem for degenerate LDB division algebras]\label{degeneratetheo}
Every degenerate LDB division algebra is weakly equivalent to an LDB division algebra of hyper-radicial type. \\
Every degenerate LDB division algebra whose attached quadratic form represents $1$ is equivalent to
an LDB division algebra of hyper-radicial type.
\end{theo}

In particular, this shows that the quadratic form attached to an LDB division algebra is either totally degenerate
or non-degenerate. Here are two nice corollaries to the above two theorems:

\begin{cor}
The dimension of an LDB division algebra is always a power of $2$.
If the field $\K$ has characteristic not $2$, then an LDB division algebra over $\K$ must have dimension $1$, $2$, $4$ or $8$.
\end{cor}

\begin{cor}\label{reversecorollary}
Let $(A,\star)$ be a division algebra. If $\star$ has a bilinear quasi-left-inversion, then $(A,\star)$
is isotopic to the reverse division algebra $(A,\star^{\text{op}})$, with $\star^{\text{op}} : (x,y) \mapsto y \star x$,
and in particular $\star^{\text{op}}$ has a bilinear quasi-left-inversion.
\end{cor}

To prove Corollary \ref{reversecorollary}, we note that the result is obvious for $1$-dimensional division algebras and that
if $(A,\star,\bullet)$ denotes an arbitrary LDB division algebra
of quadratic, quaternionic, octonionic or hyper-radicial type, then $\star$ is isotopic to $\star^{\text{op}}$.
Thus, by Theorems \ref{nondegeneratetheo} and \ref{degeneratetheo}, this property holds for all LDB division algebras.

\vskip 3mm
There is an additional powerful result whose statement encompasses both degenerate and non-degenerate LDB division algebras:

\begin{theo}\label{classbyquadformtheo}
Two LDB division algebras are equivalent if and only if their attached quadratic forms are equivalent. \\
Two LDB division algebras are weakly equivalent if and only if their attached quadratic forms are similar.
\end{theo}

\begin{Rem}\label{equivalenttoweaklyequivalent}
In Theorem \ref{classbyquadformtheo}, the direct implications are already known, and only the converse implications are non-trivial.
Moreover, the second statement in this theorem is actually an easy consequence of the first one.
Assume indeed that the first one holds, and let $(A,\star,\bullet)$ and $(B,\star',\bullet')$ be LDB division algebras
with associated quadratic forms $q_A$ and $q_B$ that are similar. Choose $\lambda \in \K^*$ such that $q_B \simeq \lambda q_A$.
Then, we note that $(A,\star,\lambda \bullet)$ is an LDB division algebra with attached quadratic form $\lambda q_A$.
From the first statement in Theorem \ref{classbyquadformtheo}, we deduce that $(A,\star,\lambda \bullet)$ is equivalent to
$(B,\star',\bullet')$, whence $(A,\star,\bullet)$ and $(B,\star',\bullet')$ are weakly equivalent.
\end{Rem}

\begin{Rem}\label{reductiontorepresent1}
Let $\lambda \in \K^*$. Then, the LDB division algebras $(A,\star,\bullet)$ and $(B,\star',\bullet')$, with respective attached quadratic forms
$q_A$ and $q_B$, are equivalent if and only if $(A,\star,\lambda\bullet)$ and $(B,\star',\lambda\bullet')$ are equivalent.
Note that the respective quadratic forms attached to the latter are $\lambda q_A$ and $\lambda q_B$.
Using this, we see that the converse implication in the first statement of Theorem \ref{classbyquadformtheo}
needs only be proved in the situation where $1$ is represented by each one of the quadratic forms attached to the LDB division algebras under consideration.
\end{Rem}

\begin{Rem}\label{isotopeHurwitz}
It is known (see e.g.\ \cite{Darpo}) that two Hurwitz algebras are isotopic if and only if their attached quadratic forms are similar
(and that two quadratic forms that are attached to Hurwitz algebras are similar if and only if they are equivalent).
Thus, in the non-degenerate case the second statement in Theorem \ref{classbyquadformtheo} is not new. However,
the first one is new and the second one is an easy corollary of it, as explained in Remark \ref{equivalenttoweaklyequivalent}.
\end{Rem}

\subsection{Some consequences}

Let $U,V,U',V'$ be finite-dimensional vector spaces.
Remember that two linear subspaces $\calS \subset \calL(U,V)$ and $\calS' \subset \calL(U',V')$ are called equivalent when
there are isomorphisms $f : U \overset{\simeq}{\rightarrow} U'$ and $g : V \overset{\simeq}{\rightarrow} V'$ such that
$\calS'=\bigl\{g \circ s \circ f^{-1} \mid s \in \calS\bigr\}$ (which amounts to saying that, in different choices of bases of the source and goal spaces,
the same space of matrices may be used to represent both $\calS$ and $\calS'$).
If $U=V$ and $U'=V'$, the endomorphism spaces $\calS$ and $\calS'$ are called similar when, in the above condition, we require that $g=f$.
In \cite{dSPLLD1}, the following (non-trivial) result was established:

\begin{theo}
Let $A$ and $B$ be LDB division algebras. Then, $\calT_A$ and $\calT_B$ are similar (respectively, equivalent)
if and only if $A$ and $B$ are equivalent (respectively, weakly equivalent).
\end{theo}

Combining this with Theorem \ref{classbyquadformtheo}, we readily deduce:

\begin{theo}
Let $A$ and $B$ be LDB division algebras, with respective quadratic forms denoted by $q_A$ and $q_B$.
The twisted operator spaces $\calT_A$ and $\calT_B$ are similar (respectively, equivalent)
if and only if the quadratic forms $q_A$ and $q_B$ are equivalent (respectively, similar).
\end{theo}

Here is a nice corollary:

\begin{cor}\label{preservercorollary}
Let $A$ and $B$ be LDB division algebras.
The following conditions are equivalent:
\begin{enumerate}[(i)]
\item There exists a rank-preserving vector space isomorphism $\Phi : \calT_A \overset{\simeq}{\rightarrow} \calT_B$.
\item The operator spaces $\calT_A$ and $\calT_B$ are equivalent.
\item The LDB division algebras $A$ and $B$ are weakly equivalent.
\end{enumerate}
\end{cor}

\begin{proof}
We already know that conditions (ii) and (iii) are equivalent, while condition (ii) obviously implies condition (i). \\
Assume that condition (i) holds.
Denote by $\varphi$ an arbitrary $2$-dimensional hyperbolic quadratic form.
Recall that the set of all singular operators of $\calT_A$ (respectively, of $\calT_B$) is the isotropy cone of a
quadratic form on $\calT_A$ (respectively, on $\calT_B$) that is equivalent to $q_A \bot \varphi$ (respectively, to $q_B \bot \varphi$).
Thus, with the quadratic Nullstellensatz, we deduce that there exists a non-zero scalar $\lambda$ such that
$q_B \bot \varphi \simeq \lambda (q_A \bot \varphi)$. As $\varphi$ is hyperbolic, we have $\lambda \varphi \simeq \varphi$, whence
$q_B \bot \varphi \simeq (\lambda q_A) \bot \varphi$ and Witt's cancellation rule yields that $q_B \simeq \lambda q_A$.
Thus, Theorem \ref{classbyquadformtheo} yields that $A$ is weakly equivalent to $B$.
\end{proof}

\subsection{Proof strategy}

Our proof has two main steps. The first one consists in the reduction to the special situation below:

\begin{Def}
Let $(A,\star,\bullet)$ be an LDB division algebra and
$e$ be a non-zero element of $A$. For $x \in A$, we set $\overline{x}:=-s_e(x)$, where $s_e$ is the reflection of $(A,q)$ along $\K e$.
We say that the LDB division algebra $(A,\star,\bullet)$ is \textbf{$e$-standard} when  the following conditions hold:
\begin{enumerate}[(i)]
\item $e \star -=\id_A$, i.e.\ $e$ is a left-sided unity for $\star$;
\item $\forall (x,y) \in A^2, \; x \bullet y=\overline{x} \star y$.
\end{enumerate}
\end{Def}

Note that $\overline{e}=e$ in this situation, whence conditions (i) and (ii) yield $q(e)=1$.
It is striking that all the special LDB division algebras we have considered in the introduction
have a left-sided unity and are standard with respect to it!

\begin{Rem}\label{standardidentityremark}
Let $(A,\star,\bullet)$ be an LDB division algebras that is standard with respect to one of its elements $e$,
and denote by $q$ its attached quadratic form and by $b_q$ its polar form.
For all $x \in A$, the identity
$$\forall y\in A, \quad x \star (-s_e(x) \star y)=q(x)\,y$$
can be restated as
\begin{equation}\label{standardidentity}
b_q(x,e)\,(x \star -)-(x \star -)^2=q(x)\,\id_A.
\end{equation}
In particular, whenever $x$ is $q$-orthogonal to $e$, we get
$$(x \star -)^2=-q(x)\,\id_A,$$
which hints to a connection with Clifford algebras.
\end{Rem}

The following result, which will be established in Section \ref{reductiontostandardsection}, allows
one to reduce the situation to the one of standard LDB division algebras:

\begin{lemma}[Standardization lemma]\label{reductionlemma}
Let $(A,\star,\bullet)$ be an LDB division algebra with attached quadratic form $q$, and let $e \in A$ be such that $q(e)=1$.
Then, there are two laws $\star'$ and $\bullet'$ on $A$ such that:
\begin{enumerate}[(i)]
\item $(A,\star',\bullet')$ is an $e$-standard LDB division algebra;
\item $(A,\star',\bullet')$ is equivalent to $(A,\star,\bullet)$;
\item $(A,\star',\bullet')$ has $q$ as its attached quadratic form.
\end{enumerate}
\end{lemma}

After we prove this, we shall analyze standard LDB division algebras by using classical structure theorems on Clifford algebras.
Firstly, we will prove that the dimension of a non-degenerate LDB division algebra must belong to $\{1,2,4,8\}$
(see Section \ref{dimensionsection}). Then, we will determine the $1$-dimensional and $2$-dimensional LDB division algebras
(Section \ref{dim1et2section}). In the remaining sections,
we shall complete the theory of LDB division algebras, first for fields of characteristic not $2$
(Section \ref{charnot2section}) and finally for fields of characteristic $2$ (Section \ref{char2section}).

\begin{Rem}\label{strategyforequivalenceremark}
Before we move forward, it is important to lay out the main technique for proving
the first statement of Theorem \ref{classbyquadformtheo}.
As we have seen, we only need to care about the case when the attached quadratic forms represent $1$.
Now, let $(A,\star,\bullet)$ and $(B,\star',\bullet')$ be LDB division algebras whose respective attached quadratic forms $q$ and $q'$
are equivalent and represent $1$. Thus, we have an isomorphism $f : A \overset{\simeq}{\rightarrow} B$
such that $q'(f(x))=q(x)$ for all $x \in A$. Then, we define two laws $\star_1$ and $\bullet_1$ on $A$
by
$$x \star_1 y:=f^{-1}(f(x) \star' f(y)) \quad \text{and} \quad x \bullet_1 y:=f^{-1}(f(x) \bullet' f(y)),$$
and we see that $(A,\star_1,\bullet_1)$ is an LDB division algebra that is equivalent to $(B,\star',\bullet')$
and whose attached quadratic form is $q$. Thus, we need only consider the case when $B=A$ and $q'=q$.

Assume now that we are in this special case. Then, we can fix a vector $e \in A$ such that $q(e)=1$, and the Standardization lemma shows that
we may assume that both $(A,\star,\bullet)$ and $(B,\star',\bullet')$ are $e$-standard.
In order to prove that these LDB division algebras are equivalent, it suffices to exhibit two automorphisms $h$ and $u$ of
the vector space $A$ such that $u$ commutes with $x \mapsto \overline{x}$ and
\begin{equation}\label{conjugateidentity}
\forall x \in A, \quad h \circ (u(x) \star -) \circ h^{-1}=x \star' -.
\end{equation}
Indeed, if we have two such automorphisms, then we obtain
$$\forall x \in A, \quad h \circ (u(x) \bullet -) \circ h^{-1}=h \circ (\overline{u(x)} \star -)\circ h^{-1}
=h \circ (u(\overline{x}) \star -)\circ h^{-1}=\overline{x} \star' -=x \bullet' -,$$
and combining this with \eqref{conjugateidentity} yields that $(A,\star,\bullet)$ is equivalent to $(B,\star',\bullet')$.
\end{Rem}

Finally, the following basic lemma will be used in a few instances.

\begin{lemma}\label{minipoly}
Let $(A,\star)$ be a division algebra and $u \in \calL(A)$ be an endomorphism that commutes with
$x \star -$ for all $x \in A$. Then, the minimal polynomial of $u$ is irreducible.
\end{lemma}

\begin{proof}
Let $p$ be an irreducible factor of the minimal polynomial of $u$. Then, $\Ker p(u)$ is stable under $x \star -$ for all $x \in A$.
On the other hand, $\Ker p(u)$ contains a non-zero vector $y$. Since $\star$ is regular,
we have $A=\{x \star y \mid x \in A\}$, whence $\Ker p(u)=A$. It follows that $p$ annihilates $u$, whence it is the minimal polynomial of $u$.
\end{proof}

\subsection{Additional notation}

Given a left vector space $V$ over a skew field $D$, we denote by $\calL(V)$ the set of all endomorphisms of $V$.
To avoid any confusion, we shall denote this set by $\calL_D(V)$ whenever necessary.

Given a quadratic form $q$ on a vector space $V$ (over a field), the Clifford algebra $C(q)$
is the quotient of the tensor algebra of $V$ by the ideal generated by the set of all elements $x \otimes x-q(x).1$ with $x \in V$.
It has a natural structure of $\Z/2$-graded algebra, and we shall denote by $C_0(q)$ its even component.

Finally, assuming that $\K$ has characteristic $2$, and given $(a,b)\in \K^2$, we denote by $[a,b]$ the quadratic form
$(x,y) \mapsto ax^2+xy+by^2$ on $\K^2$.

\section{The reduction to standard LDB division algebras}\label{reductiontostandardsection}

The purpose of the section is to prove Lemma \ref{reductionlemma}, thereby limiting the study
of LDB division algebras to the one of standard LDB division algebras.
This non-trivial result will be obtained by refining some techniques that were featured
in Section 5.3 of \cite{dSPLLD1}.

\subsection{The key lemma}

In \cite[Corollary 5.13]{dSPLLD1}, we have shown that every LDB division algebra $(A,\star,\bullet)$ is equivalent to its opposite algebra\footnote{To see that $(A,\bullet,\star)$ is an LDB division algebra, one uses the fact that $\star$ is regular to
deduce identity $\forall (x,y)\in A^2, \; x \bullet (x \star y)=q(x)\, y$
from identity $\forall (x,y)\in A^2, \; x \star (x \bullet y)=q(x)\, y$ applied to the pair $(x,x \star y)$.}$(A,\bullet,\star)$.
Here, we shall refine this statement as follows:

\begin{prop}\label{keylemma}
Let $(A,\star,\bullet)$ be an LDB division algebra with attached quadratic form $q$, and let
$s$ be a reflection of the quadratic space $(A,q)$. Then, there are
automorphisms $g$ and $h$ of $A$ such that
$$\forall (x,y)\in A^2, \quad x \bullet y=h(s(x) \star g(y)).$$
\end{prop}

The proof has three steps. Before we explain them, some additional notation is required.
In the rest of the section, we fix an LDB division algebra $(A,\star,\bullet)$ with attached quadratic form denoted by $q$.

\begin{Not}
Given an anisotropic vector $x \in A \oplus \K^2$, we denote by $s_x$ the reflection of $(A \oplus \K^2,\widetilde{q})$ along $\K x$,
and by $\pi(x)$ the projection of $x$ on $A$ along $\K^2$.
\end{Not}

Our starting point is the following result, which was the first step in the proof of Lemma 5.11 of \cite{dSPLLD1}:

\begin{lemma}\label{prelimkeylemma1}
Let $a \in (A\oplus \K^2) \setminus \K^2$ be an anisotropic vector.
Then, there are automorphisms $G$ and $F$ of $A^2$ such that
$$\forall x \in A \oplus \K^2, \quad
\Gamma_A(s_a(x),-)=G \circ \Gamma_A(s_{\pi(a)}(x),-) \circ F.$$
\end{lemma}

From there, we obtain the following generalization:

\begin{lemma}\label{prelimkeylemma2}
Let $a_1,\dots,a_p$ be anisotropic vectors of $(A \oplus \K^2) \setminus \K^2$.
Set $u:=s_{a_1}\circ \cdots \circ s_{a_p}$ and $v:=s_{\pi(a_1)}\circ \cdots \circ s_{\pi(a_p)}$.
Then, there are automorphisms $G$ and $F$ of $A^2$ such that
$$\forall x\in A \oplus \K^2, \quad
\Gamma_A(u(x),-)=G \circ \Gamma_A(v(x),-) \circ F.$$
\end{lemma}

\begin{proof}[Proof of Lemma \ref{prelimkeylemma2}]
We prove the result by induction on $p$, the case $p=1$ being given by Lemma \ref{prelimkeylemma1}.
Assume that $p>1$. Set $w:=s_{a_1}\circ \cdots \circ s_{a_{p-1}}$
and $h:=s_{\pi(a_1)}\circ \cdots \circ s_{\pi(a_{p-1})}$.
By induction, there are automorphisms $F$ and $G$ of $A^2$ such that
$$\forall x \in A \oplus \K^2, \quad \Gamma_A(w(x),-)=G \circ \Gamma_A(h(x),-) \circ F.$$
We endow $B:=A$ with a new structure of LDB division algebra:
for $(x,y)\in B^2$, we set
$$x \star' y:=h(x) \star y \quad \text{and} \quad x \bullet' y:=h(x) \bullet y.$$
Noting that $h$ is the identity on $\K^2$, we obtain
$$\forall x \in A \oplus \K^2, \quad \Gamma_B(x,-)=\Gamma_A(h(x),-),$$
whence
$$\forall x \in A \oplus \K^2, \quad \Gamma_A(w(x),-)=G \circ \Gamma_B(x,-) \circ F.$$
As $h$ induces an orthogonal automorphism of $(A,q)$, the quadratic form attached to $B$ is $q$.
In particular, for every anisotropic vector $x$ of $A \oplus \K^2$, we see that $s_x$ is the reflection of
$(B\oplus \K^2, \widetilde{q})$ along  $\K x$. Applying Lemma \ref{prelimkeylemma1} to the LDB division algebra $B$,
we obtain automorphisms $F'$ and $G'$ of $A^2$ such that
$$\forall x\in A \oplus \K^2, \quad
\Gamma_B(s_{a_p}(x),-)=G' \circ \Gamma_B(s_{\pi(a_p)}(x),-) \circ F'.$$
Therefore, for all $x \in A \oplus \K^2$, we conclude that
\begin{align*}
\Gamma_A(w(s_{a_p}(x)),-) & = G \circ \Gamma_B(s_{a_p}(x),-) \circ F \\
& = G \circ G' \circ \Gamma_B(s_{\pi(a_p)}(x),-) \circ F' \circ F \\
& = (G \circ G') \circ \Gamma_A \bigl((h \circ s_{\pi(a_p)})(x),-\bigr) \circ (F' \circ F).
\end{align*}
As $G \circ G'$ and $F' \circ F$ are automorphisms of $A^2$, this completes our inductive proof.
\end{proof}

Now, we can prove Proposition \ref{keylemma}:

\begin{proof}[Proof of Proposition \ref{keylemma}]
Denote by $B:=(A,\bullet,\star)$ the opposite LDB division algebra
of $(A,\star,\bullet)$.
Setting $T : (y,z) \in A^2 \mapsto (z,y)$ and
$t : x+(\lambda,\mu) \in A \oplus \K^2 \longmapsto x+(\mu,\lambda)$,
we obtain
$$\forall x \in A \oplus \K^2, \quad \Gamma_B(x,-)=T \circ \Gamma_A(t(x),-) \circ T^{-1}.$$
Note that $t$ is the reflection of $(A \oplus \K^2,\widetilde{q})$ along
$b:=(1,-1)$. Now, let $a \in A \setminus \{0\}$. Setting $d:=a+(1,0)$, we see that
$\widetilde{q}(d)=q(a)$, that $d \not\in \K^2$ and that $b$ is not $\widetilde{q}$-orthogonal to $d$.
In particular, $d$ is anisotropic and $t=s_b=s_d^{-1} \circ s_{s_d(b)} \circ s_d=s_d \circ s_{s_d(b)} \circ s_d$.
As $b$ is not $\widetilde{q}$-orthogonal to $d$, we see that $\pi(s_d(b))$ is a non-zero scalar multiple of $a$.
Thus, $s_{\pi(d)}=s_a=s_{\pi(s_d(b))}$. Applying Lemma \ref{prelimkeylemma2}, we obtain automorphisms
$G$ and $F$ of $A^2$ such that
$$\forall x \in A \oplus \K^2, \quad
\Gamma_A(t(x),-)=G \circ \Gamma_A\bigl((s_{\pi(d)} \circ s_{\pi(s_d(b))}\circ s_{\pi(d)})(x),-\bigr) \circ F.$$
Set $G_1:=T \circ G$ and $F_1:=F \circ T^{-1}$.
As $s_{\pi(d)} \circ s_{\pi(s_d(b))}\circ s_{\pi(d)}=(s_a)^3=s_a$, we obtain
\begin{equation}\label{ident1}
\forall x \in A \oplus \K^2, \quad
\Gamma_B(x,-)=G_1 \circ \Gamma_A(s_a(x),-) \circ F_1.
\end{equation}
As $s_a$ leaves $(1,0)$ and $(0,1)$ invariant, we have in particular
\begin{equation}\label{identitepart}
\begin{cases}
\Gamma_B\bigl((1,1),-\bigr) & =G_1 \circ \Gamma_A\bigl((1,1),-\bigr) \circ F_1 \\
\Gamma_B\bigl((1,0),-\bigr) & =G_1 \circ \Gamma_A\bigl((1,0),-\bigr) \circ F_1 \\
\Gamma_B\bigl((0,1),-\bigr) & =G_1 \circ \Gamma_A\bigl((0,1),-\bigr) \circ F_1.
\end{cases}
\end{equation}
As $\Gamma_B\bigl((1,1),-\bigr)=\id_{A^2}=\Gamma_A\bigl((1,1),-\bigr)$, the first identity in \eqref{identitepart} yields $G_1=F_1^{-1}$.
As $\Ker \Gamma_A\bigl((0,1),-\bigr)=A \times \{0\}=\Ker \Gamma_B\bigl((0,1),-\bigr)$
and $\Ker \Gamma_A\bigl((1,0),-\bigr)=\{0\} \times A=\Ker \Gamma_B\bigl((1,0),-\bigr)$,
the second and third identities in \eqref{identitepart} yield that $F_1$ stabilizes $A \times \{0\}$ and $\{0\} \times A$, giving rise to
automorphisms $g$ and $h$ of $A$ such that
$$\forall (y,z) \in A^2, \; F_1(y,z)=(g(y),h(z)).$$
Thus, for all $x \in A$, applying \eqref{ident1} to $x$ yields
$$\forall (y,z)\in A^2, \; (x \bullet z, x \star y)=\Bigl(g^{-1}\bigl(s_a(x) \star h(z)\bigr),h^{-1}\bigl(s_a(x) \bullet g(y)\bigr)\Bigr).$$
Extracting the first components from both sides concludes the proof.
\end{proof}

\subsection{Completing the reduction to the standard case}

Now, we are ready to prove Lemma \ref{reductionlemma}.
Let $(A,\star,\bullet)$ be an LDB division algebra with attached quadratic form $q$,
and let $e \in A$ be such that $q(e)=1$. Denote by $b_q$ the polar form of $q$.
For $x \in A$, we set $\overline{x}:=-s_e(x)$. Since $q(e)=1$, we have
$$\forall x \in A, \quad \overline{x}=-x+b_q(x,e)\,e.$$
Proposition \ref{keylemma} yields automorphisms $h$ and $g$ of $A$ such that
$$\forall (x,y)\in A^2, \quad x \bullet y=h(\overline{x} \star g(y)).$$
For $x \in A$, denote by $M(x)$ the endomorphism $y \in A \mapsto x \star y$ of $A$.
With the above identity, we deduce that
$$\forall x \in A, \quad M(x)\circ h \circ M(\overline{x}) \circ g=q(x)\id_A.$$
Setting $N(x):=M(x) \circ h$ and $f:=g^{-1}\circ h$, we deduce that
$$\forall x \in A, \quad N(x)\circ N(\overline{x})=q(x)\,f,$$
and in particular $N(e)^2=f$.
Let us prove that $N(x)$ commutes with $N(e)$ for all $x$ in $A$.
Fix $x \in A$. Polarizing the above quadratic identity in $(x,e)$ yields
$$N(x)\circ N(\overline{e})+N(e)\circ N(\overline{x})=b_q(x,e)\,f.$$
As $\overline{e}=e$, this reads
$$N(x)\circ N(e)-N(e)\circ N(x)+b_q(x,e)\,N(e)^2=b_q(x,e)\,f.$$
Since $N(e)^2=f$, we deduce, as claimed, that
$$N(x)\circ N(e)-N(e)\circ N(x)=0.$$
It ensues that
\begin{equation}\label{derivedidentity}
\forall x \in A, \; \bigl(N(e)^{-1}\circ  N(x)\bigr)\circ \bigl(N(e)^{-1}\circ N(\overline{x})\bigr)=N(e)^{-2} \circ N(x) \circ N(\overline{x})
=q(x)\,\id_A.
\end{equation}
We are ready to conclude. Defining new laws $\star'$ and $\bullet'$ on $A$ by
$$x \star' y:=N(e)^{-1}(x \star h(y)) \quad \text{and} \quad x \bullet' y:=h^{-1}(x \bullet N(e)[y]),$$
we see that $(A,\star',\bullet')$ is an LDB division algebra that is equivalent to $(A,\star,\bullet)$ and whose attached quadratic form is $q$.
Let us check that $(A,\star',\bullet')$ is $e$-standard.
First, we note that $x \star '-=N(e)^{-1}\circ N(x)$ for all $x \in A$, and in particular $e\star' -=\id_A$.
Next, identity  \eqref{derivedidentity} yields that the law $\bullet''$ defined by
$x \bullet'' y:=\overline{x} \star' y$ is a bilinear quasi-left-inversion of $\star'$, and $q$
is the quadratic form attached to the LDB division algebra $(A,\star',\bullet'')$. Thus, we have a scalar
$\lambda$ such that $\bullet''=\lambda \bullet'$, and $q=\lambda q$ since $\lambda q$ is the quadratic form attached to
$(A,\star',\lambda \bullet')$. It follows that $\lambda=1$ and $\bullet''=\bullet'$ and hence $(A,\star',\bullet')$
is $e$-standard, as claimed. This completes the proof of Lemma \ref{reductionlemma}.

\vskip 3mm
The following result was already proved in \cite{dSPLLD1} by using the fact that an invertible alternating matrix must have
an even number of columns. Here, we use the Standardization lemma to give a new proof of it:

\begin{cor}\label{dimensioncor}
Let $(A,\star,\bullet)$ be an LDB division algebra with dimension $n>1$.
Then, $n$ is even.
\end{cor}

\begin{proof}
It suffices to deal with the case when $n>2$.
Using Remark \ref{reductiontorepresent1} and the Standardization lemma, we see that no generality is lost in assuming that
$(A,\star,\bullet)$ is $e$-standard for some $e \in A \setminus \{0\}$. Then, as $n>2$, we can choose a vector
$x \in A \setminus \K e$ that is $q$-orthogonal to $e$. Thus, the endomorphism $f:=x \star -$ satisfies $f^2=-q(x)\id_A$, and
$\langle 1,q(x)\rangle=\langle q(e),q(x)\rangle$ is anisotropic since it is equivalent to a subform of $q$.
It follows that $-q(x)$ is a non-square in $\K$, whence the polynomial $p(t):=t^2+q(x)$ is irreducible over $\K$.
Thus, $p(t)$ is both irreducible and the minimal polynomial of $f$, whence the dimension of $A$ is a multiple of the degree of $p(t)$.
\end{proof}

\section{The dimension of a non-degenerate LDB division algebra}\label{dimensionsection}

Let $(A,\star,\bullet)$ be a non-degenerate LDB division algebra with dimension $n$
and attached quadratic form $q$. We wish to show that $n \in \{1,2,4,8\}$, thereby proving the first statement in Theorem
\ref{nondegeneratetheo}. As we already know from Corollary \ref{dimensioncor} that $n$ is even or equals $1$,
we can simply assume that $n \geq 6$, in which case we know that $n$ is even and we need to prove that $n=8$.

Using Remark \ref{reductiontorepresent1} and Lemma \ref{reductionlemma},
we can further assume that $A$ is standard with respect to one of its non-zero elements $e$.
Then, as $q$ is non-degenerate, we can find a linear subspace $V \subset E$ with codimension $2$
such that $q_V$ is non-degenerate and $V \bot \K e$.
By \eqref{standardidentity}, we have
$\forall x \in V, \; (x \star -)^2=-q(x)\id_A$, whence the linear map $x \mapsto x \star -$ extends into
a homomorphism of $\K$-algebras from the Clifford algebra
$C(-q_V)$ to $\calL(A)$.
As the dimension of $V$ is even and $-q_V$ is non-degenerate, the algebra $C(-q_V)$ is simple
(see \cite[Chapter 9]{Scharlau} Theorem 2.10 for fields of characteristic not $2$, and Corollary 4.7 for fields of characteristic $2$).
In particular, the above homomorphism is injective, whence $\dim C(-q_V) \leq \dim \calL(A)$, which leads to
$2^{n-2} \leq n^2$. Obviously, this yields $n \leq 8$.

Assume now that $n=6$. In that case, we note that the above homomorphism yields a structure of
left $C(-q_V)$-module on $A$. As $C(-q_V)$ is a finite-dimensional simple $\K$-algebra,
it is isomorphic to $\Mat_p(\L)$ for some skew field extension $\L$ of $\K$ and some positive integer $p$,
and all the minimal left $C(-q_V)$-modules have dimension $p\,d$ over $\K$, where $d:=[\L : \K]$.
In particular, $\dim A$ should be a multiple of $p\,d$. As on the other hand
$p^2d=\dim C(-q_V)$, we deduce that $p d$ is a power of $2$ and that $pd\geq \sqrt{\dim C(-q_V)}=4$.
As $p d$ divides $6$, this is absurd.
Therefore, $n=8$, which completes the proof of the first statement in Theorem \ref{nondegeneratetheo}.

\section{LDB division algebras of dimension at most $2$}\label{dim1et2section}

Our aim in this short section is to understand the structure of all LDB division algebras with dimension at most $2$.

Let $(A,\star,\bullet)$ be an LDB division algebra with dimension at most $2$ and attached quadratic form denoted by $q$.

Assume first that $(A,\star,\bullet)$ has dimension $1$.
Without loss of generality, we may assume that $A=\K$.
If $A$ is $1$-standard, it is obvious that $(A,\star,\bullet)=(\K,\cdot,\cdot)$.
Thus, Lemma \ref{reductionlemma} yields that $(A,\star,\bullet)$ is equivalent to $(\K,\cdot,\cdot)$ whenever $q$ represents $1$.
Using Remark \ref{reductiontorepresent1}, we deduce that Theorem \ref{classbyquadformtheo} holds for all $1$-dimensional LDB division algebras.

The following lemma deals with $2$-dimensional LDB division algebras:

\begin{lemma}\label{2dimlemma}
Let $(A,\star,\bullet)$ be a $2$-dimensional LDB division algebra whose attached quadratic form $q$ represents $1$.
Then, $A$ is equivalent to the quadratic LDB division algebra associated with the field extension $C_0(q)$ of $\K$.
\end{lemma}

\begin{proof}
We choose $e \in A$ such that $q(e)=1$.
Using the Standardization lemma, we see that no generality is lost in assuming that $e \star -=\id_A$.
In that situation, we see that $\L:=\{x \star - \mid x \in A\}$ is a two-dimensional subspace of $\calL(A)$
that contains $\id_A$ and in which all the non-zero operators are invertible. Choosing $f \in \L \setminus \K \id_A$, we find that
$\L=\K[f]$, whence $\L$ is a field extension of $\K$.
Then, using the isomorphisms $F : x \in A \mapsto (x \star -) \in \L$ and
$G : g \in \L \mapsto g(e)\in A$, one easily checks that
$$\forall (x,y)\in A^2, \; G (F(x) \circ G^{-1}(y))=x \star y,$$
and hence $\L$ is weakly equivalent to $A$ through a weak equivalence that maps $e$ to $1_\L$.
This yields a scalar $\lambda$ such that $(A,\star,\lambda \bullet)$ is equivalent to $\L$, with
$\lambda q(e)=N_{\L/\K}(1_\L)$ and $\lambda q \simeq N_{\L/\K}$. It follows that $\lambda=1$,
that $A$ is equivalent to $\L$ and that $N_{\L/\K}$ is equivalent to $q$.

If $q \simeq \langle 1,-a\rangle$ for some $a \in \K$,
then $C_0(q) \simeq \K[\sqrt{a}] \simeq \L$.
If $\K$ has characteristic $2$ and $q \simeq [1,a]$ for some $a \in \K$, then
$X^2+X+a$ is irreducible over $\K$ since $q$ is anisotropic, and
we see that $C_0(q) \simeq \K[t]/(t^2+t+a) \simeq \L$.
As $q$ represents $1$, this shows that $C_0(q)$ is always equivalent to $\L$.

We conclude by noting that isomorphic quadratic extensions of $\K$ obviously yield equivalent LDB division algebras.
\end{proof}

Combining Lemma \ref{2dimlemma} with Remark \ref{reductiontorepresent1} yields that every
$2$-dimensional LDB division algebra is weakly equivalent to a quadratic LDB division algebra,
and we obtain both statements in Theorem \ref{classbyquadformtheo} for $2$-dimensional LDB division algebras.

\section{LDB algebras over fields of characteristic not $2$}\label{charnot2section}

In this section, we assume that the underlying field $\K$ does not have characteristic $2$, and
we obtain Theorems \ref{nondegeneratetheo} and \ref{classbyquadformtheo} in this situation.
We will understand the structure of all $4$-dimensional LDB division algebras over $\K$, and then of all $8$-dimensional LDB division algebras over $\K$.

In each case, the basic strategy is to consider an LDB division algebra whose attached quadratic form represents $1$
and to prove that this quadratic form is a Pfister form. Then, we use the strategy outlined in Remark \ref{strategyforequivalenceremark}
to prove that two LDB division algebras are equivalent whenever their attached quadratic forms are equivalent.

\subsection{Four-dimensional LDB division algebras}

Let $(A,\star,\bullet)$ be an LDB division algebra with dimension $4$ and whose attached quadratic form $q$ represents $1$.
Our aim is to prove that $(A,\star,\bullet)$ is equivalent to any quaternionic LDB division algebra whose
attached quadratic form is equivalent to $q$.
We choose $e \in A$ such that $q(e)=1$.
By the Standardization lemma, we lose no generality in further assuming that $A$ is $e$-standard.

Our first step is the following:

\begin{claim}\label{4pfistercarnot2}
The quadratic form $q$ is a Pfister form, that is $q \simeq \langle 1,-a\rangle \otimes \langle 1,-b\rangle$ for some $(a,b)\in \K^2$.
\end{claim}

\begin{proof}
As $q$ represents $1$, it suffices to prove that the discriminant of $q$ equals $1$.
Assume on the contrary that this is not the case.

Denote by $V$ the orthogonal complement of $\K e$ in $(A,q)$, so that $\overline{x}=-x$ for all $x \in V$.
Thus, $(x \star -)^2=-q(x)\id_A$ for all $x \in V$.
The map $x \in V \mapsto (x \star -) \in \calL(A)$ can then be extended into a homomorphism of $\K$-algebras
$\Phi : C(-q_V) \longrightarrow \calL(A)$.
Choosing $(a,b,c)\in \K^3$ such that $q_V \simeq \langle a,b,c\rangle$, we know that
$C(-q_V) \simeq C\langle-ab,-ac\rangle_\K \otimes_\K \L$ where $\L:=\K[t]/(t^2-abc)$.
On the other hand, $q \simeq \langle 1,a,b,c\rangle$.
Thus, $abc$ is not a square in $\K$, and hence $\L$ is a quadratic extension of $\K$.
Thus, $C(-q_V) \simeq C\langle-ab,-ac\rangle_\L$ and $C(-q_V)$ is a simple $\K$-algebra with dimension $8$.
Using the above homomorphism of $\K$-algebras, we obtain a structure of left $C(-q_V)$-module on $A$.
As $A$ has dimension $4$ over $\K$, the algebra $C\langle-ab,-ac\rangle_\L$ is not a skew field, which yields that
$\langle 1,ab,ac,bc\rangle_\L \simeq \langle 1,-(-ab),-(-ac),(-ab)(-ac)\rangle_\L$ is hyperbolic. As $abc$ is a square in $\L$, it follows that
$\langle 1,ab(abc),ac(abc),bc(abc)\rangle_\L \simeq \langle 1,c,b,a\rangle_\L$ is hyperbolic.
However, as $\L=\K[t]/(t^2-abc)$ and $q$ is anisotropic, this would yield scalars $a'$ and $b'$ in $\K$ such that
$q\simeq \langle 1,c,b,a\rangle \simeq \langle 1,-abc\rangle \otimes \langle a',b'\rangle$
(by Remark 5.11 of \cite[Chapter 2]{Scharlau})
and one would conclude that the discriminant of $q$ equals $1$. This contradicts our initial assumption, completing the proof.
\end{proof}

Now, we have found non-zero scalars $a$ and $b$ such that $q \simeq \langle 1,-a\rangle \otimes \langle 1,-b\rangle$.
On the other hand, the quaternionic LDB division algebra associated with the quaternion algebra $C\langle a,b\rangle$
has its attached quadratic form equivalent to $\langle 1,-a\rangle \otimes \langle 1,-b\rangle$.
In order to conclude that this quaternionic LDB division algebra is equivalent to $(A,\star,\bullet)$, it suffices to prove
that every LDB division algebra whose attached quadratic form is equivalent to $q$ is equivalent to $(A,\star,\bullet)$ itself.
Using the strategy outlined in Remark \ref{strategyforequivalenceremark}, we find that it suffices to prove the following result:

\begin{claim}
Let $(B,\star',\bullet')$ be an $e$-standard LDB division algebra (with $B=A$)
with attached quadratic form $q$. Then, there are automorphisms $h$ and $u$ of $A$ such that $u$ commutes with $x \mapsto \overline{x}$
and $\forall x \in A, \; h \circ (u(x) \star -) \circ h^{-1}=x \star' -$.
\end{claim}

\begin{proof}
Denote by $V$ the orthogonal complement of $\K e$ in $(A,q)$, so that $-q_V$ has discriminant $1$.
As in the proof of Claim \ref{4pfistercarnot2}, we find that the linear map $x \in V \mapsto (x \star -) \in \calL(A)$
extends into a homomorphism of $\K$-algebras $\Phi : C(-q_V) \rightarrow \calL(A)$.
However, as $-q_V$ has discriminant $1$, we obtain $C(-q_V) \simeq C_0(-q_V) \times C_0(-q_V)$, whose center $\calZ$ is isomorphic to
$\K \times \K$. Let $p$ be an idempotent of $\calZ$.
Then, $\Phi(p)$ is an idempotent of $\calL(A)$ that commutes with $x \star -$ for all $x \in A$
(as $e \star -=\id_A$), and hence Lemma \ref{minipoly} yields $\Phi(p)=\id_A$ or $\Phi(p)=0$.
Varying $p$ and taking linear combinations, it follows that $\Phi$
maps every element of the center of $C(-q_V)$ to a scalar multiple of $\id_A$.

Next, by Witt's cancellation rule, we can choose an orthogonal basis $(e_1,e_2,e_3)$ of $V$ such that
$q(e_1)=-a$, $q(e_2)=-b$ and $q(e_3)=ab$. Classically, $e_1e_2e_3$ belongs to the center of $C(-q_V)$, and
$(e_1e_2e_3)^2=q(e_1)q(e_2)q(e_3)=(ab)^2$. Thus, $\Phi(e_1e_2e_3)=\pm ab \,\id_A$,
which yields an $i \in \{0,1\}$ such that
$$(e_3 \star -)=(-1)^i\,(e_1\star -)\circ (e_2 \star -).$$

Next, we set $P:=\Vect(e_1,e_2)$, and we extend the linear map $x \in P \mapsto (x \star -) \in \calL(A)$
into a homomorphism $\Psi : C(-q_P) \rightarrow \calL(A)$ of $\K$-algebras. This homomorphism is injective because $C(-q_P)$
is a simple algebra. Denote by $v$ the vector space isomorphism from $A$ to $C(-q_P)$ that maps $e_1$ to $e_1$, $e_2$ to $e_2$,
$e_3$ to $e_1e_2$ and $e$ to $1$, and denote by $s$ the reflection of $(A,q)$ along $\K e_3$.
Then, as $e \star -=\id_A$ and $e_3 \star -=(-1)^i \Psi(e_1e_2)$, we obtain
$$\forall x \in A, \quad x \star -=\Psi(v(s^i(x))).$$

Now, we are close to the conclusion. With the same line of reasoning applied to $(A,\star',\bullet')$, we find
a homomorphism $\Psi' : C(-q_P) \rightarrow \calL(A)$ of $\K$-algebras together with a $j \in \{0,1\}$ such that
$$\forall x \in A, \quad x \star' -=\Psi'(v(s^j(x))).$$
As $C(-q_P)$ is a simple algebra and $\calL(A)$ is a central simple $\K$-algebra, the Skolem-Noether theorem
\cite[Chapter 8, Theorem 4.2]{Scharlau}\footnote{Instead of the Skolem-Noether theorem,
one could simply note that $\Psi$ and $\Psi'$ define two structures of left $C(-q_P)$-vector space on $A$,
both with dimension $1$ over $C(-q_P)$, and one could conclude by choosing an isomorphism $h$ from the first structure to the second one.}
yields an automorphism $h$ of $A$ such that $\forall y \in C(-q_P), \; \Psi'(y)=h \circ \Psi(y) \circ h^{-1}$.
Thus, for all $x \in A$, we obtain
$$x \star' -=h \circ (s^{j-i}(x) \star -) \circ h^{-1}.$$
The endomorphism $s$ commutes with $x \mapsto \overline{x}$ because $e$ is $q$-orthogonal to $e_3$.
Thus, $s^{j-i}$ commutes with $x \mapsto \overline{x}$ and the claimed result is proved.
\end{proof}

From there, the arguments from Remark \ref{strategyforequivalenceremark} show that
an LDB division algebra is equivalent to $(A,\star,\bullet)$ whenever its attached quadratic form is equivalent to $q$.
In particular, $(A,\star,\bullet)$ is equivalent to a quaternionic LDB division algebra.
More precisely, if $q \simeq \langle 1,-a\rangle \otimes \langle 1,-b\rangle$, then $(A,\star,\bullet)$ is equivalent to $C\langle a,b\rangle$.

\subsection{Eight-dimensional LDB division algebras}\label{dim8charnot2}

Let $(A,\star,\bullet)$ be an LDB division algebra with dimension $8$ and whose attached quadratic form $q$ represents $1$.
We choose $e \in A$ such that $q(e)=1$.
Our aim is to prove that $q$ is a Pfister form and that every LDB division algebra with attached quadratic form $q$
is equivalent to $(A,\star,\bullet)$. By the Standardization lemma, it suffices to do this when $A$ is $e$-standard
and, in what follows, we shall assume that this condition holds.

As $q$ is anisotropic and $\K$ has characteristic not $2$, the subspace $V:=\{e\}^\bot$
satisfies $A=V \oplus \K e$.

We start with a preliminary result on the quadratic form $q$:

\begin{claim}
The quadratic form $q$ has discriminant $1$.
\end{claim}

\begin{proof}
By \eqref{standardidentity}, we have
 $\forall x \in V, \; (x \star -)^2=-q(x)\,\id_A$ and hence the linear map $x \in V \mapsto (x \star -)  \in \calL(A)$
can be extended into a homomorphism of $\K$-algebras $\Phi : C(-q_V) \rightarrow \calL(A)$.
This homomorphism cannot be injective since $\dim \calL(A)=8^2<2^7=\dim C(-q_V)$. Therefore, $C(-q_V)$ is not simple,
which entails that the discriminant of $-q_V$ equals $1$. It follows that the discriminant of $q$ equals $1$, as claimed.
\end{proof}

Next, we can choose an orthogonal basis $(e_i)_{1 \leq i \leq 7}$ of $V$ in which
$q(e_7)=\underset{i=1}{\overset{6}{\prod}} q(e_i)$. We set $W:=\Vect(e_i)_{1 \leq i \leq 6}$.
The restriction of $\Phi$ to $C(-q_W)$ must be injective since $C(-q_W)$ is simple. As $\dim C(-q_W)=2^6=\dim \calL(A)$,
we deduce that $\Phi$ is surjective. In particular, since
$z:=e_1e_2e_3e_4e_5e_6e_7$ lies in the center of $C(-q_V)$, the element $\Phi(z)$ is a scalar multiple of the identity,
whence we have a scalar $\lambda$ such that
$\Phi(e_7)=\lambda \underset{i=1}{\overset{6}{\prod}} \Phi(e_i)$.
Using $\Phi(e_i)^2=-q(e_i)\,\id_A$ for all $i \in \lcro 1,6\rcro$ together with the fact that $\Phi(e_1),\dots,\Phi(e_6)$
are pairwise skew-commuting operators,
we deduce that $\lambda^2 \underset{i=1}{\overset{6}{\prod}} q(e_i)=q(e_7)$, and we conclude that $\lambda=\pm 1$.
Using this, we shall prove:

\begin{claim}
An LDB division algebra is equivalent to $(A,\star,\bullet)$ whenever its attached quadratic form equals $q$.
\end{claim}

\begin{proof}
As explained in Remark \ref{strategyforequivalenceremark}, it suffices to consider an $e$-standard LDB division algebra $(A,\star',\bullet')$
with $q$ as its attached quadratic form, and to exhibit automorphisms $h$ and $u$ of $A$ such that
$$\forall x\in A, \quad x \star' -=h \circ (u(x) \star -) \circ h^{-1}$$
and $u$ commutes with $x \mapsto \overline{x}$.
Using the above considerations, we see that the linear maps $x \in W \mapsto (x \star -) \in \calL(A)$ and
$x \in W \mapsto (x \star' -) \in \calL(A)$ extend, respectively, into homomorphisms of $\K$-algebras
$\Psi : C(-q_W) \rightarrow \calL(A)$ and $\Psi' : C(-q_W) \rightarrow \calL(A)$.
Moreover, we have indexes $i$ and $j$ in $\{0,1\}$ such that
$$e_7 \star -=(-1)^i \underset{k=1}{\overset{6}{\prod}} \Psi(e_k) \quad \text{and} \quad
e_7 \star' -=(-1)^j \underset{k=1}{\overset{6}{\prod}} \Psi'(e_k).$$
Denote by $v$ the linear map from $A$ to $C(-q_W)$ that maps $e_k$ to itself for all $k \in \lcro 1,6\rcro$,
that maps $e$ to $1$ and that maps $e_7$ to $e_1e_2e_3e_4e_5e_6$. Denote finally by $s$ the reflection of $(A,q)$
along $\K e_7$. Then, we obtain
$$\forall x \in A, \quad x \star -=\Psi(v(s^i(x))) \quad \text{and} \quad
x \star' -=\Psi'(v(s^j(x))).$$
Since $C(-q_W)$ is a simple $\K$-algebra and $\calL(A)$ is a central simple $\K$-algebra, the Skolem-Noether theorem
yields an automorphism $h$ of $A$ such that $\forall y \in C(-q_W), \; \Psi'(y)=h \circ \Psi(y) \circ h^{-1}$.
Therefore,
$$\forall x \in A, \quad x \star' -=h \circ (s^{j-i}(x) \star -) \circ h^{-1},$$
which is the desired conclusion since $s$ commutes with $x \mapsto \overline{x}$ (because $e$ is $q$-orthogonal to $e_7$).
\end{proof}

Now, coming back to the structure of $(A,\star,\bullet)$, we aim at proving that $q$ is a Pfister form.

\begin{claim}\label{dim8subpfisterclaim}
One of the $4$-dimensional subforms of $q$ is a Pfister form.
\end{claim}

\begin{proof}
We consider the $2$-dimensional space $P:=\Vect(e_1,e_2)$.
Since $\forall x \in P, \; (x \star -)^2=-q(x)\,\id_A$, the linear map $x \in P \mapsto (x \star -) \in \calL(A)$
can be extended into a homomorphism $\varphi : C(-q_P) \rightarrow \calL(A)$ of $\K$-algebras. The quadratic form
$\langle 1\rangle \bot q_P$ is anisotropic since it is equivalent to a subform of $q$, whence the quaternion algebra
$\calH:=C(-q_P)$ is a skew field. Using $\varphi$, we naturally endow $A$ with a structure of left vector space over $\calH$.
Now, set $U:=\Vect(e_i)_{3 \leq i \leq 7}$. By polarizing the quadratic identity
$$\forall x \in V, \quad (x \star -)^2=-q(x)\,\id_A,$$
we obtain that the operator $x \star -$ skew-commutes with both $e_1 \star -$ and $e_2 \star -$ for all $x \in U$.
Denoting by $\gamma$ the involution of $C(-q_P)$ associated with its $\Z/2$-graduation, we deduce that
the map $x \star -$ is semi-$\calH$-linear with associated field automorphism $\gamma$ for all $x \in U$
(that is, for all $h \in \calH$, all $x \in U$ and all $y \in A$, we have $x \star (h.y)=\gamma(h).(x\star y)$).
It follows that, for all $x \in A$, the endomorphism $x \star -$ splits as
$y \mapsto h.y+u(y)$, where $h \in \calH$ and $u \in \calL_\K(A)$ is semi-$\calH$-linear with associated field automorphism $\gamma$.

We know that $x \mapsto x \star e$ is an automorphism of the $\K$-vector space $A$.
Thus, the space $E:=\{x \in A : \; x \star e \in \calH\,e\}$ is a $4$-dimensional linear subspace of $A$ over $\K$, and
it contains $e$, obviously. As $\overline{x}=-x+b_q(x,e)\,e$ for all $x \in A$, we deduce that $E$ is stable under $x \mapsto \overline{x}$.
Let $x \in E$. We contend that $x \star -$ stabilizes $\calH e$.
Indeed, we have a quaternion $h \in \calH$ and a semi-$\calH$-linear endomorphism $u$ of $A$
such that $\forall y \in A, \; x \star y = h.y+u(y)$. As $h.e \in \calH e$ and $x \star e \in \calH\,e$, we obtain
$u(e) \in \calH\,e$. As $u$ is semi-$\calH$-linear, it follows that $u(\calH\,e) \subset \calH\,e$, which entails that
$x \star -$ stabilizes $\calH e$: we denote by $\kappa(x)$ the endomorphism of the $\K$-vector space $\calH e$ induced by $x \star -$.
Thus, we have $\kappa(x)\circ \kappa(\overline{x})=q(x)\,\id_{\calH e}$ for all $x \in E$.
As $E$ and $\calH e$ are both $4$-dimensional vector spaces over $\K$, we can choose an isomorphism
of $\K$-algebras $\Delta : \calL(\calH e) \overset{\simeq}{\longrightarrow} \calL(E)$.
Then, we define two laws $\star_1$ and $\bullet_1$ on $E$ by
$$x \star_1 y:=\Delta(\kappa(x))[y] \quad \text{and} \quad x \bullet_1 y:=\Delta(\kappa(\overline{x}))[y],$$
and one easily checks that $(E,\star_1,\bullet_1)$ is an LDB division algebra with dimension $4$ and attached quadratic form
$q_E$. Since $e \in E$, the quadratic form $q_E$ represents $1$ and hence the structure theorem for $4$-dimensional LDB division algebras
shows that $q_E$ is a Pfister form, which proves our claim.
\end{proof}

\begin{claim}\label{dim8lastclaimcharnot2}
$q$ is a Pfister form.
\end{claim}

\begin{proof}
We know that, for some $(a,b)\in (\K^*)^2$, the form $\langle 1,-a,-b,ab\rangle$ is equivalent to a subform of $q$.
Using Witt's cancellation rule, we deduce that $\langle -a,-b,ab\rangle$ is equivalent to a subform of $q_V$.
Thus, no generality is lost in assuming that $q(e_1)=-a$, $q(e_2)=-b$ and $q(e_3)=ab$.
We set $F:=\Vect(e_4,e_5,e_6,e_7)$. Our aim is to prove that $q_F$ is similar to $\langle 1,-a\rangle \otimes \langle 1,-b\rangle$.

Set $P':=\Vect(e_1,e_2,e_3)$. Again, the linear map $x \in P'\mapsto (x \star -) \in \calL(A)$
is naturally extended into a homomorphism $\Delta : C(-q_{P'}) \rightarrow \calL(A)$ of $\K$-algebras.
The even subalgebra $\calQ:=C_0(-q_{P'})$ is isomorphic to the quaternion algebra $C\langle a,b\rangle$, which is
a skew field since $\langle 1,-a,-b\rangle$, being equivalent to a subform of $q$, is anisotropic.
Again, we use $\Delta$ to endow $A$ with a structure of left $\calQ$-vector space.

Polarizing the identity $\forall x \in V, \; (x \star -)^2=-q(x)\,\id_A$
yields that $x \star -$ skew-commutes with $y \star -$ for all $x \in F$ and all $y \in P'$, whence
$x \star -$ commutes with $\Delta(y)$ for all $x \in F$ and all $y \in \calQ$. In other words
the map $x \star -$ is an endomorphism of the $\calQ$-vector space $A$ for all $x \in F$.

From there, we can extend the $\K$-linear map $x \in F \mapsto (x \star -) \in \calL_\calQ(A)$ into a
homomorphism of $\K$-algebras $\Psi : C(-q_F) \longrightarrow \calL_\calQ(A)$.
As $C(-q_F)$ is simple (because $\dim F$ is even and $q_F$ is non-degenerate), we see that $\Psi$ is injective.
On the other hand, as $q \simeq q_F \bot (\langle 1,-a\rangle \otimes \langle 1,-b\rangle)$ and $q$ has discriminant $1$,
we find that $q_F$ has discriminant $1$, and hence $-q_F$ has discriminant $1$.
It follows that the center of $C_0(-q_F)$ is isomorphic to $\K \times \K$ (see \cite[Chapter 9, Theorem 2.10]{Scharlau}).
Remember that, as $\dim F$ is even, the conjugation by a non-zero element of $F$ always induces
the non-identity automorphism of the center of $C_0(-q_F)$.
Let us fix a non-trivial idempotent $p$ in the center of $C_0(-q_F)$; then we see that
$x p x^{-1}=1-p$ for all $x \in F \setminus \{0\}$.
In particular, $\Psi(p)$ is a non-trivial idempotent of $\calL_\calQ(A)$ and, for all
$x \in F$, the map $x \star -$ swaps the $\calQ$-linear subspaces $A_1:=\Ker \Psi(p)$ and $A_2:=\im \Psi(p)$.

Now, we fix an arbitrary operator $f \in \Psi(F) \setminus \{0\}$. We choose a basis $\bfB_1$ of the $\K$-vector space $A_1$ and we see that $\bfB_2:=f(\bfB_1)$
is a basis of the $\K$-vector space $A_2$, whence $\bfB:=\bfB_1 \coprod \bfB_2$ is a basis of the $\K$-vector space $A$.
For $x \in A$, let us denote by $M(x)$ the matrix of the endomorphism $x \star -$ in the basis $\bfB$.
In particular, for $x_0 \in F$ such that $f=x_0 \star -$,
we have
$$M(x_0)=\begin{bmatrix}
0 & -q(x_0) I_4 \\
I_4 & 0
\end{bmatrix}.$$
Fix $h \in \calQ$.
As $A_1$ and $A_2$ are $\calQ$-vector spaces, we know that the matrix of $y \mapsto h.y$ in $\bfB$ has the form
$$N(h)=\begin{bmatrix}
h_1 & 0 \\
0 & h_2
\end{bmatrix} \quad \text{for some $(h_1,h_2)\in \Mat_4(\K)^2$.}$$
As this matrix commutes with $M(x_0)$, a straightforward computation shows that $h_1=h_2$.
Now, for $h \in \calQ$, we can write
$$N(h)=\begin{bmatrix}
R(h) & 0 \\
0 & R(h)
\end{bmatrix},$$
so that $R(\calQ) \subset \Mat_4(\K)$ is a quaternion algebra that is isomorphic to $\calQ$.

Next, for $x \in F$, we know that $x \star -$ swaps $A_1$ and $A_2$, whence we have matrices $B(x)$ and $C(x)$ of $\Mat_4(\K)$ such that
$$M(x)=\begin{bmatrix}
0 & C(x) \\
B(x) & 0
\end{bmatrix},$$
and, as $M(x)$ commutes with $N(q)$ for all $q \in \calQ$, we see that $B(x)$ commutes with $R(q)$ for all $q \in \calQ$.
Note that $x \mapsto B(x)$ is one-to-one since $M(x)$ is non-singular for all $x \in F \setminus \{0\}$.
However, as $R(\calQ)$ is a $4$-dimensional skew field extension of $\K$, its centralizer in $\Mat_4(\K)$
is a skew field that is isomorphic to the opposite algebra $R(\calQ)^{\text{op}}$, whence
$B(F)$ equals this centralizer. As every Clifford algebra possesses an anti-automorphism, we conclude that
$B(F)$ is itself a subalgebra of $\Mat_4(\K)$ that is isomorphic to $\calQ$ as a $\K$-algebra.
For $N \in B(F)$, let us denote by $N^*$ the conjugate of $N$ in the quaternion algebra $B(F)$.

From the identity $\forall x \in F, \; (x \star -)^2=-q(x)\,\id_A$, we deduce
$$\forall x \in F, \quad B(x)C(x)=-q(x)\,I_4.$$
Now, we choose an isomorphism $\Theta : \Mat_4(\K) \rightarrow \calL(F)$ of $\K$-algebras, and we define two laws
$\star_2$ and $\bullet_2$ on $F$ as
$$x \star_2 y:=\Theta(B(x))[y] \quad \text{and} \quad  x \bullet_2 y:=\Theta(C(x))[y],$$
so that $(F,\star_2,\bullet_2)$ is an LDB division algebra with attached quadratic form $-q_F$.
However, with the law $\bullet_3$ defined on $F$ as
$$x \bullet_3 y:=\Theta(B(x)^*)[y],$$
we obtain that $(F,\star_2,\bullet_3)$ is an LDB division algebra whose attached quadratic form is $x \mapsto N_{B(F)}(B(x))$,
where $N_{B(F)}$ denotes the norm of the quaternion algebra $B(F)$. It follows that
$N_{B(F)}$ is similar to $-q_F$. As $B(F)$ is isomorphic to $\calQ$, its norm is equivalent to $\langle 1,-a\rangle \otimes \langle 1,-b\rangle$,
which yields a scalar $c \in \K$ such that $-q_F \simeq c\,\langle 1,-a\rangle \otimes \langle 1,-b\rangle$.
Finally, as $A=\Vect(e,e_1,e_2,e_3) \overset{\bot}{\oplus} F$ and $q_{\Vect(e,e_1,e_2,e_3)} \simeq \langle 1,-a\rangle \otimes \langle 1,-b\rangle$,
we conclude that $q \simeq \langle 1,-a\rangle \otimes \langle 1,-b\rangle \otimes \langle 1,-c\rangle$, as claimed.
\end{proof}

From the above claims, we conclude that $(A,\star,\bullet)$ is equivalent to the octonionic LDB division algebra associated with
the quaternion algebra $C\langle a,b\rangle$ and the scalar $c$.

\section{LDB division algebras over fields of characteristic $2$}\label{char2section}

Throughout the section, we assume that the underlying field $\K$ has characteristic $2$.
Section \ref{degeneratesection} is devoted to the proof of Theorem \ref{degeneratetheo}, that is the determination of the structure of
degenerate LDB division algebras over $\K$. The next two sections are devoted to the theory of
non-degenerate LDB division algebras with dimensions $4$ and $8$, respectively.

\subsection{Degenerate LDB division algebras}\label{degeneratesection}

In this section, we prove Theorem \ref{degeneratetheo}.
Let $(A,\star,\bullet)$ be a degenerate LDB division algebra with attached quadratic form $q$.
If $\dim A=1$, we already know from Section \ref{dim1et2section} that $(A,\star,\bullet)$ is weakly equivalent to the hyper-radicial LDB division algebra
$\K$, and that it is even equivalent to it whenever $q$ represents $1$. Thus, in the rest of the section, we assume that
$\dim A \geq 2$, to the effect that $\dim A$ is even (see Corollary \ref{dimensioncor}).
We denote by $R$ the radical of the polar form of $q$, and we split $A=R \oplus V$, so that $q_V$ is non-degenerate and $\dim V$ is even.
Thus, $\dim R \geq 2$.

Until further notice, we assume that there is an element $e \in R$ such that $q(e)=1$.
We can use the Standardization lemma to reduce the situation to the one where $(A,\star,\bullet)$
is $e$-standard. However, as $e$ belongs to the radical of $q$, formula \eqref{standardidentity} yields
$$\forall x \in A, \; (x \star -)^2=q(x)\,\id_A.$$
By polarizing this formula, we learn in particular that, for all $x \in R$, the operator $x \star -$
commutes with all the operators $y \star -$ with $y \in A$.

Now, denote by $\L$ the subalgebra of $\calL(A)$ generated by the operators $x \star -$ with $x \in R$,
so that $\L$ is a commutative subalgebra of $\calL(A)$ and every element of $\L$
commutes with all the operators $y \star -$ with $y \in A$. It follows from Lemma \ref{minipoly} that every non-zero operator in $\L$
is non-singular, and hence $\L$ is a field! Moreover, as $(x \star -)^2\in \K \id_A$ for all $x \in R$ and as $\K$ has characteristic $2$,
we find that $f^2 \in \K \id_A$ for all $f \in \L$. Thus, $\L$ is a hyper-radicial extension of $\K$.
Now, we have a natural structure of $\L$-vector space on $A$; note that
the operators $x \star -$, for $x \in A$, are all $\L$-linear.
Setting $d:=[\L:\K]$, $m:=\dim_\K V$ and $n:=\dim_\K A$, we have $\dim_\L \calL_\L(A)=\frac{n^2}{d^2}$
and $d \geq \dim R=n-m \geq 2$.

\begin{claim}
The quadratic form $q$ is totally degenerate.
\end{claim}

\begin{proof}
Assume that the contrary holds, that is $m \geq 2$.
With the identity $\forall x \in V, \; (x \star -)^2=q(x)\,\id_A$, we may
extend $x \in V \mapsto (x \star -) \in \calL_\L(A)$ into a homomorphism of $\L$-algebras
$\Phi : C((q_V)_\L) \rightarrow \calL_\L(A)$. Since $(q_V)_\L$ is non-degenerate, the algebra
$C((q_V)_\L)$ is simple, whence $\Phi$ is injective. Comparing the dimensions over $\L$ leads to $2^m \leq \frac{n^2}{d^2}$, and hence
$$(d+m)^2 \geq 2^m d^2.$$
Noting that the function $t \mapsto \frac{(t+1)^2}{t^2}$ is decreasing on the interval $(0,+\infty)$,
we find that $(k+1)^2<2 k^2$ for all $k>2$; if $d>2$, it follows
that the sequence $\Bigl((d+k)^2(2^k d^2)^{-1}\Bigr)_{k \geq 0}$ is decreasing, and as the initial value
of that sequence is $1$, the only remaining option is that $d=2$.
In that case, we see that $\Bigl((d+k)^2(2^k d^2)^{-1}\Bigr)_{k \geq 2}$ is decreasing with initial value $1$, whence the above inequality
yields $m=2$. Thus, $\dim R=\dim V=2$, and if we write $q_R \simeq \langle 1,\delta \rangle$, then
$\L$ is isomorphic to the inseparable quadratic extension $\K[t]/(t^2-\delta)$.
As $\dim_\L C((q_V)_\L)=4=\dim_\L \calL_\L(A)$, we deduce that the above homomorphism $\Phi$
is an isomorphism of $\L$-algebras, whence the quaternion algebra $C((q_V)_\L)$ is not a skew field, and we deduce that the quadratic form
$\langle 1 \rangle \bot q_V$ becomes isotropic over $\L$.

We shall conclude by showing that this contradicts the assumption that
$q \simeq \langle 1,\delta \rangle \bot q_V$ should be anisotropic.
Denote by $t$ an element of $\L \setminus \K$ such that $t^2=\delta$.
We embed $V$ naturally into $\L \otimes_\K V$. As $\langle 1 \rangle \bot q_V$ becomes isotropic over $\L$, we
find $\alpha \in \{0,1\}$ together with a non-zero pair $(x,y)\in V^2$
such that $(q_V)_\L(x+t y)=\alpha$. This yields $q(x)+\delta q(y)=\alpha$ and $b_q(x,y)=0$.
As $V$ has dimension $2$, the second equality shows that $x$ and $y$ are collinear, yielding
a non-zero vector $z \in V$ together with a non-zero pair $(\lambda,\mu)\in \K^2$ such that $x=\lambda z$ and $y=\mu z$.
Thus, $\alpha=(\lambda^2+\delta \mu^2)q(z)$. As $(\lambda,\mu) \neq (0,0)$ and $q$ is anisotropic, we have $\lambda^2+\delta \mu^2 \neq 0$.
Setting $\beta:=\alpha (\lambda^2+\delta \mu^2)^{-1} \lambda$ and
$\gamma:=\alpha (\lambda^2+\delta \mu^2)^{-1} \mu$, we finally obtain
$$q(z)=\beta^2 +\delta \gamma^2$$
by noting that $\alpha^2=\alpha$.
As $z \neq 0$, this shows that $\langle 1,\delta \rangle \bot q_V$ is isotropic, contradicting our assumptions.
Thus, we conclude that $m=0$, as claimed.
\end{proof}

Next, we prove that $(A,\star,\bullet)$ is equivalent to a hyper-radicial LDB division algebra.
Since $A$ is a non-zero vector space over $\L$, we have $d \leq n$, and hence the injective linear map
$\Phi : x \in A \mapsto (x \star -) \in \L$ is an isomorphism of vector spaces (over $\K$).
It follows that $q$ is equivalent to the quadratic form $\lambda \mapsto \lambda^2$ on $\L$, which is
associated with the hyper-radicial LDB division algebra $\L$.

In order to conclude, it remains to prove that an LDB division algebra is equivalent to $(A,\star,\bullet)$ whenever
its attached quadratic form is equivalent to $q$. As explained in Remark \ref{strategyforequivalenceremark},
it suffices in this prospect to consider an LDB division algebra $(A,\star',\bullet')$ that is $e$-standard and whose attached quadratic form is $q$.
Using the isomorphism $\Phi$, we enrich the $\K$-vector space $A$ into a hyper-radicial field extension of $\K$, so that $\Phi$
is an isomorphism of $\K$-algebras. Then, the newly defined multiplication $\times$ on $A$
only depends on $q$, as we have
$$\forall (x,y)\in A^2, \; q(x \times y).\id_A=\Phi(x \times y)^2=\Phi(x)^2\Phi(y)^2=q(x)q(y)\,\id_A,$$
whence
$$\forall (x,y)\in A^2, \; q(x \times y)=q(x)q(y).$$
As $q$ is injective (it is a group homomorphism from $(A,+)$ to $(\K,+)$ since $q$ is totally degenerate, and its kernel is zero since $q$ is anisotropic),
the above identity shows that $\times$ is determined by $q$.
Now, with the new LDB division algebra $(A,\star',\bullet')$, we obtain another homomorphism of $\K$-algebras
$\Phi' : A \rightarrow \calL_\K(A)$. As $A$ is a field, the Skolem-Noether theorem
yields an automorphism $h$ of the $\K$-vector space $A$
such that
$$\forall x \in A, \quad \Phi'(x)=h \circ \Phi(x) \circ h^{-1},$$
which reads
$$\forall x \in A, \quad x \star' -=h \circ (x \star -) \circ h^{-1}.$$
From Remark \ref{strategyforequivalenceremark}, we conclude that $(A,\star',\bullet')$ is equivalent to $(A,\star,\bullet)$.
In particular, we obtain that $(A,\star,\bullet)$ is equivalent to the hyper-radicial LDB division algebra $(A,\times,\times)$.

\vskip 3mm
Now, we can conclude: given a degenerate LDB division algebra $(A,\star,\bullet)$ with attached quadratic form $q$,
we can choose a non-zero vector $e$ in the radical of $b_q$, and we find a scalar $\lambda$ such that
the quadratic form attached to
$(A,\star,\lambda\bullet)$ maps $e$ to $1$. With the above results, we deduce that $(A,\star,\lambda\bullet)$
is equivalent to a hyper-radicial LDB division algebra. Finally, using Remark \ref{strategyforequivalenceremark},
we obtain the remaining results of Theorem \ref{degeneratetheo},
together with both statements of Theorem \ref{classbyquadformtheo} for degenerate LDB division algebras.

\subsection{Four-dimensional non-degenerate LDB division algebras}\label{dim4char2}

Let $(A,\star,\bullet)$ be a non-degenerate LDB division algebra with dimension $4$.
We assume that the quadratic form $q$ attached to $A$ represents $1$, and we choose $e \in A$ such that $q(e)=1$.
Our goal is to show that $(A,\star,\bullet)$ is equivalent to a quaternionic LDB division algebra and that
an LDB division algebra is equivalent to $(A,\star,\bullet)$ whenever its attached quadratic form is equivalent to $q$.
In this prospect, we lose no generality in assuming that $A$ is $e$-standard.
Our first step consists in proving that the Arf invariant of $q$, which we classically denote by $\Delta(q)$, equals $0$.

Given a non-degenerate alternating form $B$ on a finite-dimensional vector space $V$, a symplectic basis of
$(V,B)$ is a basis $(e_1,\dots,e_{2n})$ of $V$ in which the $2$-dimensional subspaces $\Vect(e_{2k-1},e_{2k})$, for $k \in \lcro 1,n\rcro$,
are pairwise $B$-orthogonal and $B(e_{2k-1},e_{2k})=1$ for all $k \in \lcro 1,n\rcro$.
Remember that, given a non-degenerate quadratic form $\varphi$ on a finite-dimensional vector space $V$,
the Arf invariant of $\varphi$ is the class of $\underset{k=1}{\overset{n}{\sum}} \varphi(e_{2k-1})\varphi(e_{2k})$ in the quotient (additive) group $\K/\{x^2+x \mid x \in \K\}$
for any symplectic basis $(e_1,\dots,e_{2n})$ of $(V,b_\varphi)$.

Let us extend $e$ into a symplectic basis $(e_1,e_2,e_3,e)$ of $(A,b_q)$.
We set $P:=\Vect(e_1,e_2)$, so that $q_P$ is a non-degenerate quadratic form.
As every vector of $P$ is orthogonal to $e$, we obtain $\forall x\in P, \; (x \star -)^2=q(x)\,\id_A$,
whence the linear map $x \in P \mapsto (x \star -) \in \calL(A)$ can be naturally extended into a homomorphism of algebras
$\Phi : C(q_P) \rightarrow \calL(A)$. Set $f:=e_3 \star -$.
By polarizing the identity
$\forall x \in A, \; (x \star -)^2=b_q(x,e)\,(x \star -)+q(x)\,\id_A$, we obtain
$$\forall x \in P, \quad \Phi(x)\circ f+f \circ \Phi(x)=\Phi(x).$$
Let $x \in P \setminus \{0\}$. Then, $\Phi(x) \circ f \circ \Phi(x)^{-1}=f+\id_A$.
However, $x(e_1e_2)x^{-1}=e_1e_2+1$ (by standard computations in the Clifford algebra $C(q_P)$).
Thus, $\Phi(x) \circ (f+\Phi(e_1e_2)) \circ \Phi(x)^{-1}=f+\Phi(e_1e_2)$.
It follows that
$$g:=f+\Phi(e_1e_2)$$
commutes with $x \star -$ for all $x \in P$.
On the other hand, we see that $\Phi(e_1e_2)\circ f \circ \Phi(e_1e_2)^{-1}=\Phi(e_1)\circ (f+\id_A) \circ \Phi(e_1)^{-1}=f$, whence
$g$ also commutes with $f=e_3 \star -$. Finally, $f+\Phi(e_1e_2)$ commutes with $\id_A=e \star -$,
whence $g$ commutes with $x \star -$ for all $x \in A$.
On the other hand, as we have just seen that $f$ and $\Phi(e_1e_2)$ commute, we obtain
$$g^2+g=f^2+f+\Phi((e_1e_2)^2+e_1e_2)=(q(e_3)+q(e_1)q(e_2))\,\id_A.$$

Assume now that the Arf invariant of $q$ is non-zero. As $q \simeq [q(e_1),q(e_2)] \bot [1,q(e_3)]$,
this invariant is represented by the scalar $\delta:=q(e_3)+q(e_1)q(e_2)$, and we deduce that the polynomial
$t^2+t+\delta$ is irreducible over $\K$. Thus, $\L:=\K[g]$ is a field and we can use it to extend the
scalar multiplication on $A$ to turn $A$ into a vector space over $\L$.
The above commutations show that $x \star -$ is $\L$-linear for all $x \in A$, thus yielding a homomorphism
$\Psi : C(q_P)_\L \rightarrow \calL_\L(A)$ of $\L$-algebras. As $\dim_\L (A)=2$, we see that $\dim_\L C((q_P)_\L)=4=\dim_\L \calL_\L(A)$.
Since $C(q_P)_\L$ is simple, we deduce that $\Psi$ is an isomorphism, to the effect that the quaternion algebra $C(q_P)_\L$
is not a skew field. It follows that the quadratic form $\langle 1\rangle \bot [q(e_1),q(e_2)]$ becomes isotropic over $\L$.
Note that $g^2+g=\delta\,\id_A$. We naturally embed $P$ into the $\L$-vector space $\L \otimes_\K P$.
Then, we find $\varepsilon \in \{0,1\}$ together with a non-zero pair $(x,y)\in P^2$ such that
$(q_P)_\L(x+g y)=\varepsilon$. As $\langle 1 \rangle\bot q_P$ is anisotropic, we see that $y \neq 0$.
Expanding the above identity, we obtain $q(x)+g b_q(x,y)+g^2 q(y)=\varepsilon$, whence
$$q(x)+\delta q(y)=\varepsilon \quad \text{and} \quad b_q(x,y)=q(y).$$
As $y \neq 0$, we have $q(y) \neq 0$ and we can set $x_1:=\frac{x}{q(y)}$, so that $(x_1,y)$ is a symplectic basis of $P$.
As $q(x_1)=\frac{q(x)}{q(y)^2}=\frac{\varepsilon}{q(y)^2}+\frac{\delta}{q(y)}$, we deduce that
$\Delta(q_P)=\bigl[\delta+\frac{\varepsilon}{q(y)}\bigr]$, whence
$\Delta(q_{\Vect(e_3,e)})=\Delta(q)-\Delta(q_P)=\bigl[\frac{\varepsilon}{q(y)}\bigr]$. As $q_{\Vect(e_3,e)}$ represents $1$, it follows that
$q_{\Vect(e_3,e)} \simeq \bigl[1,\frac{\varepsilon}{q(y)}\bigr]$, whence
$$q \simeq \bigl[q(x_1),q(y)\bigr] \bot \Bigl[1,\frac{\varepsilon}{q(y)}\Bigr].$$
From there, we deduce that $\Bigl\langle q(y),\frac{\varepsilon}{q(y)}\Bigr\rangle \simeq \langle q(y),\varepsilon q(y)\rangle$ is equivalent to a subform of $q$,
which is absurd because this form is obviously isotropic.
Thus, we have obtained:
\begin{center}
The Arf invariant of $q$ equals $0$.
\end{center}
As $q$ represents $1$, it follows that $q$ is equivalent to $[1,ab] \bot [a,b]$ for some $(a,b)\in (\K^*)^2$,
that is $q$ is equivalent to the norm of the quaternion algebra $C [a,b]$.
In order to conclude, it suffices to prove that two LDB division algebras are equivalent whenever their attached quadratic forms are equivalent to $q$.

To see this, we start by noting that the polynomial $t^2+t+(q(e_3)+q(e_1)q(e_2))$ splits over $\K$.
However, by Lemma \ref{minipoly}, the minimal polynomial of $g$ must be irreducible, whence it must have degree $1$ since it divides
$t^2+t+(q(e_3)+q(e_1)q(e_2))$. This yields a scalar $\lambda$ such that $g=\lambda \id_A$, that is
$$e_3 \star -=\Phi(e_1e_2)+\lambda\,\id_A.$$

Now, let $(A,\star',\bullet')$ be an $e$-standard LDB division algebra with attached quadratic form $q$.
As above, the linear map $x \in P \mapsto (x \star' -)\in \calL(A)$ extends into a homomorphism $\Phi' : C(q_P) \rightarrow \calL(A)$ of $\K$-algebras,
and we obtain a scalar $\mu \in \K$ such that
$$e_3 \star' -=\Phi'(e_1e_2)+\mu \id_A.$$
Since $C(q_P)$ is simple and $\calL(A)$ is central and simple, the Skolem-Noether theorem yields an automorphism $h$ of $A$
such that $\Phi'(y)=h \circ \Phi(y) \circ h^{-1}$ for all $y \in C(q_P)$.
Next, for $\alpha \in \K$, we denote by $u_\alpha$ the isomorphism from $A$ to $C(q_P)$ that maps
each vector $e_1,e_2$ to itself and that maps $e_3$ to $\alpha +e_1e_2$ and $e$ to $1$.
Thus, we have
$$\forall x \in A, \quad x \star -=\Phi(u_\lambda(x)) \quad \text{and} \quad x \star' -=\Phi'(u_\mu(x)).$$
Therefore, with $v:=u_{\lambda}^{-1} \circ u_\mu$, we see that
$$\forall x \in A, \quad x \star' -=h \circ (v(x) \star -) \circ h^{-1}.$$
Finally, we note that $v(x)=x+(\mu-\lambda)b_q(x,e)e$ for all $x \in A$, whence
$v$ commutes with $x \mapsto \overline{x}=x+b_q(x,e)e$ (both being polynomials in the operator $x \mapsto b_q(x,e)\,e$).
Thus, with Remark \ref{strategyforequivalenceremark}, we conclude that $(A,\star,\bullet)$ and $(A,\star',\bullet')$ are equivalent.

Finally, as the quadratic form attached to $C[a,b]$ is equivalent to $q$,
we conclude that this quaternionic LDB division algebra is equivalent to $(A,\star,\bullet)$, which completes the proof.

\subsection{Eight-dimensional non-degenerate LDB division algebras}\label{dim8char2}

Here, we determine the non-degenerate LDB division algebras of dimension $8$ over $\K$.
The strategy is largely similar to the one of Section \ref{dim8charnot2}.
First of all, we consider an $8$-dimensional non-degenerate LDB division algebra $(A,\star,\bullet)$
whose attached quadratic form $q$ represents $1$.
We need to prove that $A$ is equivalent to an octonionic LDB division algebras and that
every LDB division algebra with an equivalent attached quadratic form is equivalent to $A$.
Choosing $e \in A$ such that $q(e)=1$, we know from the Standardization lemma that no generality is lost in assuming that
$(A,\star,\bullet)$ is $e$-standard.

We extend $e$ into a symplectic basis $(e_1,\dots,e_7,e)$ of $(A,b_q)$. Set $f:=e_7 \star -$,
$W:=\Vect(e_1,\dots,e_6)$, and note that $q_W$ is non-degenerate and that every vector of $W$ is orthogonal to $e$, so that
$$\forall x \in W, \quad (x \star -)^2=q(x)\,\id_A.$$
Thus, the linear map $x \in W \mapsto (x \star -) \in \calL(A)$ extends into a homomorphism
$\Phi : C(q_W) \rightarrow \calL(A)$ of $\K$-algebras. As $C(q_W)$ is simple and $\dim C(q_W)=2^6=\dim \calL(A)$,
we find that $\Phi$ is an isomorphism.

Let $x \in W \setminus \{0\}$. Polarizing the identity $\forall y \in \Vect(e_1,\dots,e_7), \; (y \star -)^2+b_q(y,e)\,(y\star -)=q(y)\,\id_A$,
we obtain $(x \star -) \circ f+f \circ (x \star -)=(x \star -)$, whence
$(x \star -) \circ f \circ (x \star -)^{-1}=f+\id_A$.
However, $x (e_1e_2+e_3e_4+e_5e_6) x^{-1}=(e_1e_2+e_3e_4+e_5e_6)+1$,
whence $x \star -$ commutes with $\Phi(e_1e_2+e_3e_4+e_5e_6)+f$.
As $\Phi$ is surjective and $\calL(A)$ is a central $\K$-algebra, we obtain a scalar $\lambda$ such that
$$\Phi(e_1e_2+e_3e_4+e_5e_6)+f=\lambda \id_A.$$
 From there, we find
\begin{align*}
 f^2+f & =\Phi((e_1e_2+e_3e_4+e_5e_6)^2+(e_1e_2+e_3e_4+e_5e_6))+(\lambda^2+\lambda)\id_A \\
 & =\bigl(q(e_1)q(e_2)+q(e_3)q(e_4)+q(e_5)q(e_6)\bigr)+\lambda^2+\lambda.
 \end{align*}
However, $f^2+f=q(e_7)\,\id_A$, whence $\Delta [1,q(e_7)]=\Delta(q_W)$. As
$A=W \overset{\bot}{\oplus} \Vect(e,e_7)$,
we conclude:
\begin{center}
The Arf invariant of $q$ equals $0$.
\end{center}

\begin{claim}
Every LDB division algebra with $q$ as its attached quadratic form is equivalent to $(A,\star,\bullet)$.
\end{claim}

\begin{proof}
Again, we simply need to consider an $e$-standard LDB division algebra $(A,\star',\bullet')$
with attached quadratic form $q$. As above, we extend $x \in W \mapsto (x \star' -) \in \calL(A)$
into an isomorphism $\Phi' : C(q_W) \rightarrow \calL(A)$ of $\K$-algebras, and we find a scalar $\lambda'$ such that
$$e_7 \star' -=\lambda'\,\id_A+\Phi'(e_1e_2+e_3e_4+e_5e_6).$$
Using the Skolem-Noether theorem, we obtain an automorphism $h$ of $A$ such that
$$\forall y \in C(q_W), \quad \Phi'(y)=h \circ \Phi(y) \circ h^{-1}.$$
Now, for $\alpha \in \K$, denote by $v_\alpha$ the linear map from $A$ to $C(q_W)$ that maps $e_i$ to itself for all $i \in \lcro 1,6\rcro$,
that maps $e$ to $1$ and that maps $e_7$ to $\alpha+e_1e_2+e_3e_4+e_5e_6$.
Thus, for all $x \in A$, we find
$$x \star -=\Phi(v_\lambda(x)) \quad \text{and} \quad
x \star'-=\Phi'(v_{\lambda'}(x)).$$
With $u:=v_\lambda^{-1} \circ v_{\lambda'}$, it follows that
$$\forall x \in A, \quad x \star' -=h \circ (u(x) \star -) \circ h^{-1}.$$
Finally, we note that $u(x)=x+(\lambda'-\lambda)b_q(x,e)e$ for all $x \in A$, and hence $u$ commutes with
$x \mapsto \overline{x}$ as both operators are polynomials in $x \mapsto b_q(x,e)\,e$.
From there, Remark \ref{strategyforequivalenceremark} entails that $(A,\star,\bullet)$ is equivalent to $(A,\star',\bullet')$.
\end{proof}

\begin{claim}\label{subpfisterclaimchar2}
There is a $4$-dimensional subspace $V_1$ of $A$ that contains $e$ and such that
$q_{V_1} \simeq [1,ab] \bot [a,b]$ for some $(a,b)\in (\K^*)^2$.
\end{claim}

\begin{proof}
Set $P:=\Vect(e_5,e_6)$.
As $\Vect(e,e_5,e_6)=\K e  \overset{\bot}{\oplus} P$, we see that $\langle 1 \rangle \bot [q(e_5),q(e_6)]$ is anisotropic,
whence the quaternion algebra $\calQ:=C(q_P)$ is a skew field. Seeing it naturally as a subalgebra of $C(q_W)$,
we can use the isomorphism $\Phi$ to endow $A$ with a structure of left vector space over $\calQ$.

Next, we analyze how the operators $x \star -$ behave with that new vector space structure.
For all $x \in \Vect(e_1,\dots,e_4)$, we obtain that $x \star -$ commutes with $e_5 \star -$ and $e_6 \star -$
by polarizing the identity $\forall y \in \Vect(e_1,\dots,e_6), \; (y \star -)^2=q(y)\,\id_A$.
Thus, $x \star -$ is $\calQ$-linear for all $x \in \Vect(e_1,\dots,e_4)$.
On the other hand, we have seen that $e_7 \star -$ is a linear combination of $\id_A$, of
$(e_1 \star -)\circ (e_2 \star -)+(e_3 \star -)\circ (e_4 \star -)$ and of
$\Phi(e_5e_6)$. Thus, for every $x \in A$, there exists $h \in \calQ$ and a $\calQ$-linear endomorphism $u$ of $A$
such that
$$\forall y \in A, \quad x \star y=h.y+u(y).$$
Then, setting $V_1:=\bigl\{x \in A : \; x \star e \in \calQ\,e\bigr\}$,
we proceed as in the proof of Claim \ref{dim8subpfisterclaim} and endow $V_1$ with a structure of LDB division algebra
with attached quadratic form $q_{V_1}$. As $V_1$ contains $e_5$ and $e_6$, the form $q_{V_1}$ is not totally degenerate,
whence Theorem \ref{degeneratetheo} yields that $q_{V_1}$ is non-degenerate, and as it represents $1$
we deduce from the results of Section \ref{dim4char2} that $q_{V_1} \simeq [1,ab] \bot [a,b]$ for some $(a,b)\in (\K^*)^2$.
\end{proof}

\begin{claim}
There is a scalar $c$ such that $q \simeq \langle 1,c \rangle \otimes \bigl([1,ab] \bot [a,b]\bigr)$.
\end{claim}

\begin{proof}
We fix a subspace $V_1$ given by Claim \ref{subpfisterclaimchar2}. Then, changing the basis if necessary, we can assume that
$V_1=\Vect(e_5,e_6,e_7,e)$ and that $q_{\Vect(e_5,e_6)} \simeq [a,b]$.
Setting $V_2:=\Vect(e_1,e_2,e_3,e_4)=V_1^\bot$, our aim is to prove that
$q_{V_2}$ is similar to $[1,ab] \bot [a,b]$. The line of reasoning is very similar to the one of the proof of Claim \ref{dim8lastclaimcharnot2}.
Firstly, since the Arf invariant of $q$ is $0$ and the one of $q_{V_1}$ is $0$,
the Arf invariant of $q_{V_2}$ is $0$.
Now, with $P:=\Vect(e_5,e_6)$ and $\calQ:=C(q_P)$, we consider again $A$ with its structure of left $\calQ$-vector space induced by $\Phi$.
For all $x \in V_2$, the map $x \star -$ is an endomorphism of the $\calQ$-vector space $A$, whence $\Phi$ induces an injective homomorphism of
$\K$-algebras from $C(q_{V_2})$ to $\calL_Q(A)$.
As the Arf invariant of $q_{V_2}$ equals zero, the center of $C_0(q_{V_2})$ is isomorphic to $\K \times \K$ whence
it contains a non-trivial idempotent $p$. Let $x \in V_2 \setminus \{0\}$. The conjugation $y \mapsto x y x^{-1}$ induces the non-identity automorphism
of the center of $C_0(q_{V_2})$, which maps $p$ to $1-p$. Setting $g:=\Phi(p)$, it follows that
$(x \star -) \circ g=(\id-g) \circ (x \star -)$, whence $x \star -$ swaps $A_1:=\Ker g$ and $A_2:=\im g$.
From there, one uses the same line of reasoning as in the proof of Claim \ref{dim8lastclaimcharnot2} to obtain that
$q_{V_2}$ is the quadratic form attached to an LDB division algebra that is weakly equivalent to the quaternionic LDB division algebra $\calQ$,
whence $q_{V_2}$ is similar to the norm of $\calQ$. This yields a scalar $c$ such that
$q_{V_2} \simeq c\,\bigl([1,ab] \bot [a,b]\bigr)$. Finally, as $A=V_1 \overset{\bot}{\oplus} V_2$, we conclude that
$q \simeq q_{V_1} \bot q_{V_2} \simeq \langle 1,c\rangle \otimes \bigl([1,ab] \bot [a,b]\bigr)$.
\end{proof}

With the above data, one concludes that $(A,\star,\bullet)$ is equivalent to the octonionic LDB division algebra
associated with the quaternion algebra $C[a,b]$ and with the scalar $c$. This completes the proof of the last point of Theorem \ref{nondegeneratetheo}
for fields of characteristic $2$. Thus, Theorems \ref{nondegeneratetheo}, \ref{degeneratetheo} and \ref{classbyquadformtheo} are now fully established.

\end{document}